\documentclass[reqno,11pt,a4paper]{amsart}

\usepackage{tikz}

\usepackage[mathscr]{eucal}
\usepackage{graphics,epic}
\usepackage{amsfonts}
\usepackage{amscd}
\usepackage{latexsym}
\usepackage{amsmath,amssymb, amsthm, stmaryrd, bm, bbm,appendix,rotating}
\usepackage[all,2cell]{xy}

\newtheorem{theorem}{Theorem}[section]
\newtheorem*{theorem*}{Theorem}

\newtheorem{lemma}[theorem]{Lemma}
\newtheorem{proposition}[theorem]{Proposition}
\newtheorem{corollary}[theorem]{Corollary}

\newtheorem*{conjecture*}{Conjecture}

\newtheorem*{question*}{Question}
\theoremstyle{remark}
\newtheorem{remark}[theorem]{Remark}

\theoremstyle{definition}
\newtheorem{definition}[theorem]{Definition}

\newcommand{\ie}{{\em i.e.~}\ }

\newcommand{\opname}[1]{\operatorname{\mathsf{#1}}}

\renewcommand{\mod}{\opname{mod}\nolimits}

\newcommand{\proj}{\opname{proj}\nolimits}

\newcommand{\pd}{\opname{pd}\nolimits}

\newcommand{\thick}{\opname{thick}\nolimits}

\newcommand{\Hom}{\opname{Hom}}
\newcommand{\End}{\opname{End}}

\newcommand{\Ext}{\opname{Ext}}

\newcommand{\ca}{{\mathcal A}}
\newcommand{\cb}{{\mathcal B}}

\newcommand{\cd}{{\mathcal D}}

\newcommand{\cf}{{\mathcal F}}

\newcommand{\cm}{{\mathcal M}}

\newcommand{\cq}{{\mathcal Q}}

\newcommand{\ct}{{\mathcal T}}

\newcommand{\cx}{{\mathcal X}}
\newcommand{\cy}{{\mathcal Y}}

\numberwithin{equation}{section}
\begin{document}

\title[silted algebras of hereditary algebras]{Classification of silted algebras for two quivers of Dynkin type $\mathbb{A}_{n}$}

\author{Zongzhen Xie, Dong Yang and Houjun Zhang}

\address{Zongzhen Xie, Department of Mathematics and Computer Science, School of Biomedical Engineering and Informatics, Nanjing Medical University, Nanjing, 211166, P. R. China.
}

\email{zzhx@njmu.edu.cn}
\address{Dong Yang, Department of Mathematics, Nanjing University, Nanjing 210093, P. R. China
}

\email{yangdong@nju.edu.cn}
\address{Houjun Zhang, School of Science, Nanjing University of Posts and Telecommunications, Nanjing 210023, P. R. China
}

\email{zhanghoujun@njupt.edu.cn}

\begin{abstract}
In this paper, we give a complete classification of silted algebras for the quiver $\overrightarrow{\mathbb{A}}_{n}$ of type $\mathbb{A}_{n}$ with linear orientation and for the quiver obtained from $\overrightarrow{\mathbb{A}}_{n}$ by reversing the arrow at the unique source. Based on the classification, we also compute the number of silted algebras for these two quivers.\\
{MSC 2020:} 16G99, 16E35, 16E45.\\
{Key words:} 2-term silting complex, silted algebra, rooted quiver with relation, tilted algebra, shod algebra.
\end{abstract}

\maketitle


\section{Introduction}\label{s:introduction}
\medskip

{\bf Silted algebras} were introduced by Buan and Zhou \cite{BuanZhou16} as endomorphisms algebras of 2-term silting complexes over finite-dimensional hereditary algebras. These algebras enjoy a very nice homological property: they are either tilted algebras (\ie endomorphism algebras of tilting modules over hereditary algebras) or strictly shod algebras (\ie shod algebras of global dimension $3$), see \cite{BuanZhou16}. Their module categories are closely related to those over hereditary algebras. Precisely, let $A$ be finite-dimensional hereditary algebra and $B$ be the endomorphism algebra of a 2-term silting complex over $A$ (we call $B$ a \emph{silted algebra} of type $A$). Then there is a torsion pair $(\ct,\cf)$ of the category $\mod A$ of finite-dimensional right $A$-modules and a split torsion pair $(\cx,\cy)$ of $\mod B$ such that $\ct$ is equivalent to $\cy$ and $\cf$ is equivalent to $\cx$, just as in classical tilting theory, see \cite{BuanZhou16,BuanZhou16b} (also \cite{XieYangZhang23}).

If $A$ is a finite-dimensional hereditary algebra of finite representation type, then there are only finitely many isomorphism classes of basic 2-term silting complexes, and this number is given in \cite{ObaidNaumanFakiehRingel14}. It is of interest to have a classification (up to isomorphism) of basic silted algebras of type $A$ as well as the strictly shod algebras among them and calculate their numbers. 
In this paper we focus on the path algebras $\Lambda_n$ and $\Gamma_n$ of {\bf two quivers of type $\mathbb{A}_n$}:
\[
\begin{xy}
(-12,0)*+{1}="1",
(0,0)*+{2}="2",
(12,0)*+{\cdots}="3",
(28,0)*+{n-1}="4",
(42,0)*+{n}="5",
\ar"2";"1", \ar"3";"2",\ar"4";"3",\ar"5";"4",
\end{xy}, \hspace{15pt}
\begin{xy}
(-12,0)*+{1}="1",
(0,0)*+{2}="2",
(12,0)*+{\cdots}="3",
(28,0)*+{n-1}="4",
(42,0)*+{n}="5",
\ar"2";"1", \ar"3";"2",\ar"4";"3",\ar"4";"5",
\end{xy}
\]
We will study $\Lambda_n$ and $\Gamma_n$ in Section~\ref{s:silted-algebra-linear-orientation} and Section~\ref{s:silted-algebra-mutated-orientation}, respectively. Precisely, we will classify the basic silted algebras and give a formula for the number of these silted algebras. The {\bf main results} are Theorem~\ref{thm:the-number-of-silted-algebras-of-Lambda_n} for $\Lambda_n$, and Theorems~\ref{thm:silted-algebras-of-Gamman} and~\ref{thm:number-of-silted-algebras-of-Gamman} for $\Gamma_n$.
As a consequence of this classification, we see that among the silted algebras of type $\Lambda_n$ and of type $\Gamma_n$ there are no strictly shod algebras. In \cite{ZhangLiu22} it is shown that this holds true for the path algebra of any quiver of type $\mathbb{A}$ with the help of geometric models. In Section~\ref{s:no-strictly-shod-algebra-of-type-A} we give an alternative proof for this result. 
In \cite{Zhang24}, the third author gives a classification of silted algebras of type the path algebras of two quivers of type $\mathbb{D}_n$ as well as a classification of the strictly shod algebras among the silted algebras.

The structure of the paper is as follows. In Section~\ref{s:prelininary}, we give the definitions and basic properties of tilted algebras and silted algebras. In Section~\ref{s:rooted-quivers-with-relation}, we study full connected rooted subquivers with relation of the genealogical tree, in particular, we introduce an operation on them to produce certain new rooted quivers with relation, which appear in the classification of silted algebras of type $\Gamma_n$. This idea of using rooted quivers with relation may also help to classify silted algebras of general type $\mathbb{A}$ and even of type $\mathbb{D}$. In Section~\ref{s:silted-algebra-linear-orientation}, we study 2-term silting complexes over $\Lambda_n$ and silted algebras of type $\Lambda_n$ and prove Theorem~\ref{thm:the-number-of-silted-algebras-of-Lambda_n}. In Section~\ref{s:silted-algebra-mutated-orientation}, we study 2-term silting complexes over $\Gamma_n$ and silted algebras of type $\Gamma_n$ and prove Theorems~\ref{thm:silted-algebras-of-Gamman} and~\ref{thm:number-of-silted-algebras-of-Gamman}. In Section~\ref{s:no-strictly-shod-algebra-of-type-A}, we give an alternative proof for the fact that there are no strictly shod algebras among silted algebras of type $\mathbb{A}$.

\medskip
\noindent\emph{Notations and conventions.} Throughout this paper $k$ is a field and all algebras are finite-dimensional $k$-algebras. For an algebra $A$ we denote by $\mod A$ the category of finite-dimensional (right) $A$-modules, by $K^{b}(\proj A)$ the bounded homotopy category of finitely generated projective $A$-modules, and by $\tau$ the Auslander--Reiten translation both in $\mod A$ and in $K^b(\proj A)$. For an $A$-module $M$, we denote by $|M|$ the number of non-isomorphic indecomposable direct summands of $M$. For an object $M$ in a triangulated category $\ct$, we denote by $\thick(M)$ the smallest triangulated subcategory of $\ct$ which contains $M$ and which is closed under taking direct summands.

For a finite set $X$, we denote by $|X|$ its cardinality. For a rational number $x$, we denote by $[x]$ the largest integer smaller than or equal to $x$.


\section{Preliminaries on tilted and silted algebras}
\label{s:prelininary}

In this section, we recall some relevant definitions and properties concerning symmetric product, tilted algebras and silted algebras. We refer the reader to \cite{AdachiIyamaReiten14,AiharaIyama12,AssemSimsonSkowronski06,BuanZhou16,KellerVossieck88}.

\subsection{Symmetric product}
Let $X$ and $Y$ be two finite sets. Denote by $X\times_s Y$ the set of all non-ordered pairs $\{x,y\}$, where $x\in X$ and $y\in Y$. 
\begin{lemma}
\label{lem:symmetric-product}
\begin{itemize}
\item[(a)] $X\times_s Y=Y\times_s X$.
\item[(b)] If $X$ and $Y$ have no intersection, then $X\times_s Y\cong X\times Y$.
\item[(c)] $|X\times_s X|=\frac{|X|(|X|+1)}{2}$.
\item[(d)] If $X'$ is a subset of $X$, then $|X'\times_s X|=|X'|\times |X|-\frac{|X'|(|X'|-1)}{2}$.
\item[(e)] Assume that $X$ and $Y$ have no intersection. If $X'$ is a subset of $X$ and $Y'$ is a subset of $Y$, then 
$|(X'\times_s Y)\cup (Y'\times_s X)|=|X'|\times |Y|+|Y'|\times |X|-|X'|\times |Y'|$.
\end{itemize}
\end{lemma}
\begin{proof} (a) and (b) are clear.

(c) Let $\leq$ be any total order on $X$. Then $X\times_s X$ is in bijection with the set of ordered pairs $(x_1,x_2)$ with $x_1\leq x_2$. The cardinality of the latter set is $\frac{|X|(|X|+1)}{2}$.

(d) We have
\[
X'\times_s X=(X'\times_s X')\sqcup (X'\times_s (X-X'))\cong (X'\times_s X')\sqcup (X'\times (X-X')),
\] 
so
\begin{align*}
|X'\times_s X|&=|X'\times_s X'|+|X'\times (X-X')|\\
&=\frac{|X'|(|X'|+1)}{2}+|X'|(|X|-|X'|)\\
&=|X'|\times |X|-\frac{|X'|(|X'|-1)}{2}
\end{align*}

(e) We have
\begin{align*}
(X'\times_s Y)\cup (Y'\times_s X)&=(X'\times_s(Y'\sqcup (Y-Y')))\cup (Y'\times_s (X'\sqcup(X-X')))\\
&=((X'\times_s Y')\sqcup(X'\times_s (Y-Y')))\cup ((Y'\times_s X')\sqcup (Y'\times_s (X-X')))\\
&=(X'\times_s Y')\sqcup ((X'\times_s (Y-Y'))\cup(Y'\times_s (X-X')))\\
&=(X'\times_s (Y-Y'))\sqcup (X'\times_s Y')\sqcup (Y'\times_s (X-X'))\\
&\cong(X'\times (Y-Y'))\sqcup (X'\times Y')\sqcup (Y'\times (X-X')).
\end{align*}
The last equality and the isomorphism hold because $X$ and $Y$ have no intersection. So
\begin{align*}
|(X'\times_s Y)\cup (Y'\times_s X)|&=|X'\times (Y-Y')|+ |X'\times Y'|+|Y'\times (X-X')|\\
&=|X'|(|Y|-|Y'|)+|X'|\times |Y'|+|Y'|(|X|-|X'|)\\
&=|X'|\times |Y|+|Y'|\times |X|-|X'|\times |Y'|.\qedhere
\end{align*}
\end{proof}

\subsection{Tilted algebras}

Let $A$ be an algebra. 
\begin{definition}
A module $T\in \mod A$ is called a {\it partial tilting module}, if it satisfies the following two conditions:
 \begin{itemize}
\item[(1)] $\pd_A T \leq 1 $.
\item[(2)] $\Ext _{A}^{1} (T,T) =0$.\
\end{itemize}
It is called a {\it tilting module} if in addition it satisfies the following condition:
 \begin{itemize}
\item[(3)] $|T|=|A|$.
\end{itemize}
\end{definition}

As a trivial example, the free $A$-module of rank $1$ is a tilting module.  If $S$ is a projective non-injective simple $A$-module, then $T=A/S\oplus \tau^{-1}S$ is a tilting module, called an \emph{APR-tilting module}. As a generalisation of tilting modules, Adachi, Iyama and Reiten \cite{AdachiIyamaReiten14} introduced $\tau$-tilting modules.

\begin{definition}
Let $T$ be an $A$-module.
\begin{itemize}
\item[(1)] $T$ is called {\it $\tau$-rigid} if $\Hom_{A}(T,\tau T)=0$, and
$T$ is called {\it $\tau$-tilting} if $T$ is $\tau$-rigid and $|T|=|A|$.
\item[(2)] $T$ is called {\it support $\tau$-tilting} if there exists an idempotent
$e$ of $A$ such that $T$ is a $\tau$-tilting $A/ \langle e \rangle$-module.
\end{itemize}
\end{definition}

\begin{definition}
Assume that $A$ is hereditary. An algebra $B$ is said to be {\it tilted} of type $A$ if there exists a tilting module $T$ over $A$ such that $B\cong\End_{A}(T)$. When $A$ is the path algebra of a quiver whose underlying graph is $\Delta$, we also say that $A$ is tilted of type $\Delta$.
\end{definition}

Assume that $A$ is hereditary.  Let $\ca_t(A)$ denote the set of isoclasses of basic tilted algebras of type $A$. If $A$ is connected, then any tilted algebra of type $A$ is also connected. If $A=A_1\times A_2$, then $\ca_t(A)=\ca_t(A_1)\times_s\ca_t(A_2)$. Moreover, if further $|A_1|\neq |A_2|$, then $\ca_t(A_1)$ and $\ca_t(A_2)$ cannot have intersection, and $\ca_t(A)\cong\ca_t(A_1)\times\ca_t(A_2)$.


\subsection{Silted algebras}

Let $A$ be an algebra. Recall that $K^b(\proj A)$ denotes the bounded homotopy category of finitely generated projective $A$-modules. Let $K^{[-1,0]}(\proj A)$ denote the full subcategory of $K^b(\proj A)$ consisting of complexes which is concentrated in degrees $-1$ and $0$. If $A=kQ$ is the path algebra of a Dynkin quiver $Q$, then the AR-quiver of $K^b(\proj A)$ is the repetitive quiver $\mathbb{Z}Q$. In this case, the AR-quiver of $\mod A$ is a full translation subquiver of the AR-quiver of $K^b(\proj A)$ and the indecomposable objects of $K^b(\proj A)$ are precisely shifts of the indecomposable $A$-modules. Here and throughout this paper we consider an $A$-module as an object of $K^b(\proj A)$ by identifying it with its projective resolution (minimal projective solution if necessary).

\begin{definition}\label{def:silting-objects}
Let $M$ be a complex in $K^{b}(\proj A)$.
 \begin{itemize}
\item[(1)] $M$ is called {\it presilting} if $\Hom_{K^{b}(\proj A)}(M,M[i])=0$ for any $i>0$.
\item[(2)] $M$ is called {\it silting} if it is presilting and $K^{b}(\proj A)=\thick(M)$.
\item[(3)] $M$ is called \emph{tilting} if $\Hom_{K^b(\proj A)}(M,M[i])=0$ for any $i\neq 0$ and $K^b(\proj A)=\thick(M)$.
\item[(4)] $M$ is called {\it 2-term} if it has non-zero terms only in degrees $-1$ and $0$, \ie $M\in K^{[-1,0]}(\proj A)$.
\end{itemize}
\end{definition}

A 2-term presilting complex $M$ in $K^b(\proj A)$ is silting if and only if $|M|=|A|$, see for example \cite[Proposition 3.14]{BruestleYang13}. Tilting modules are 2-term tilting complexes. More generally, 2-term silting complexes and support $\tau$-tilting modules are closely related.

\begin{theorem}[{\cite[Theorem 3.2]{AdachiIyamaReiten14}}]
\label{thm:correspondence-between-silting-and-support-tau-tilting}
Taking the $0$-th cohomology defines a bijection from the set of isomorphism classes of 2-term silting complexes over $A$ to the set of isomorphism classes of support $\tau$-tilting $A$-modules.
\end{theorem}

In \cite{BuanZhou16}, Buan and Zhou introduced silted algebras.

\begin{definition}\label{def:silted-algebras}
Assume that $A$ is hereditary. An algebra $B$ is said to be \emph{silted} of type $A$ if there exists a 2-term silting complex $M$ over $A$ such that $B\cong\End_{K^{b}(\proj A)} (M)$.
\end{definition}

Some examples of silted algebras of path algebras of Dynkin quivers are given in \cite{Xing24}.

\begin{theorem}[{\cite[Theorem 0.2]{BuanZhou16}}]
\label{thm: silted-algebras-to-tilted-and-shod-algebras}
Let $B$ be a connected algebra. Then $B$ is a silted algebra if and only if $B$ is either a
tilted algebra or a strictly shod algebra, that is, $B$ has global dimension 3 and any $B$-module has projective or injective dimension no greater than 1. 
\end{theorem}

The following slightly extended version of \cite[Corollary 3.6]{BuanZhou16} is an interesting corollary of Theorem~\ref{thm: silted-algebras-to-tilted-and-shod-algebras}.

\begin{corollary}
\label{cor:endomorphism-algebra-of-2-term-tilting-is-tilted}
Assume that $A$ is hereditary. Let $T$ be a 2-term tilting complex and $B=\End_{K^b(\proj A)}(T)$. Then $B$ is a tilted algebra of type $A'$ for some hereditary algebra $A'$ which is derived equivalent to $A$.
\end{corollary}
\begin{proof}
By \cite[Section 12.5]{GabrielRoiter92} (see also \cite[Proposition 3.5]{BuanZhou18}), the global dimensional of $B$ is at most $2$. Since $B$ is a silted algebra, it follows by Theorem~\ref{thm: silted-algebras-to-tilted-and-shod-algebras} that $B$ is tilted, say of type $A'$. Then $A$, $B$ and $A'$ are derived equivalent.
\end{proof}



\section{Rooted quivers with relation}
\label{s:rooted-quivers-with-relation}

In this section we study full connected rooted subquivers with relation of the genealogical tree, which, by the work \cite[(4.1)]{HappelRingel81} of Happel and Ringel, is closely related to tilting modules over the path algebra of the quiver of type $\mathbb{A}_n$ with linear oritentation. The combinatorics of these rooted quivers with relation itself is also very interesting.
\medskip

A \emph{rooted quiver with relation} is a quiver with relation together with a vertex of the quiver, which is called the \emph{root}. By the path algebra of a rooted quiver with relation we mean the path algebra of the underlying quiver with relation, that is, the quotient of the path algebra of the underlying quiver modulo the ideal generated by the relations. A \emph{rooted subquiver with relation} is a subquiver with relation which contains the root, and it is \emph{full} if its subquiver is a full subquiver and its relations consists of all relations involving the subquiver. 
One interesting example is
the following ``genealogical" tree with $\bullet$ as the root
$$
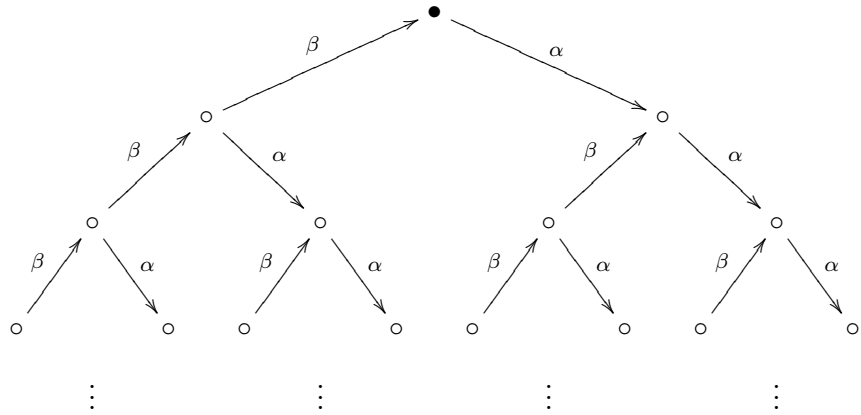
\begin{xy}
(0,0) *+{\bullet}="1",
(-30,-14) *+{\circ}="2",
(30,-14) *+{\circ}="3",
(-45,-28) *+{\circ}="4",
(-15,-28) *+{\circ}="5",
(15,-28) *+{\circ}="6",
(45,-28) *+{\circ}="7",
(-55,-42) *+{\circ}="8",
(-35,-42) *+{\circ}="9",
(-25,-42) *+{\circ}="10",
(-5,-42) *+{\circ}="11",
(5,-42) *+{\circ}="12",
(25,-42) *+{\circ}="13",
(35,-42) *+{\circ}="14",
(55,-42) *+{\circ}="15",
(-45,-50) *+{\vdots}="17",
(-15,-50) *+{\vdots}="18",
(15,-50) *+{\vdots}="19",
(45,-50) *+{\vdots}="20",
\ar^{\beta}"2";"1",\ar^{\alpha}"1";"3",\ar^{\beta}"4";"2",\ar^{\alpha}"2";"5",\ar^{\alpha}"3";"7",\ar^{\beta}"6";"3",\ar^{\beta}"8";"4",\ar^{\alpha}"4";"9",\ar^{\beta}"10";"5",\ar^{\alpha}"5";"11",\ar^{\beta}"12";"6",
\ar^{\alpha}"6";"13",\ar^{\beta}"14";"7",\ar^{\alpha}"7";"15",\end{xy}$$
with all possible relations $\alpha\beta=0$. We are interested in its rooted subquivers with relation.

\smallskip
Let $n$ be a positive integer. 
\subsection{Full connected rooted subquivers with relation of the genealogical tree}
Put 
\begin{align*}
\cq(n)&=\{\text{full connected rooted subquivers with relation of the genealogical tree}\\
&\hspace{20pt} \text{ with $n$ vertices}\}.
\end{align*}
Notice that elements of $\cq(n)$ are pairwise non-isomorphic as rooted quivers with relation, but they can be isomorphic as quivers with relation. Moreover, the path algebras of two elements of $\cq(n)$ are isomorphic if and only if these two elements are isomorphic as quivers with relation. The following is an easy observation.

\begin{lemma}
\label{lem:vertices-lying-on-a-relation}
Let $R\in\cq(n)$.
\begin{itemize}
\item[(a)] The root of $R$ is a source with one neighbour, or a sink with one neighbour, or lies on a relation and has two neighbours.
\item[(b)] For any full connected subquiver with relation of $R$, there exists a vertex such that with this vertex as the root it belongs to $\cq(m)$, where $m$ is the number of vertices.
\item[(c)] A vertex of $R$ lies on a relation if and only if it has three neighbours or it is the root and has two neighbours.
\end{itemize}
\end{lemma}

Put
\begin{align*}
\cq_h(n)&=\{R\in\cq(n)\mid R \text{ has trivial relation}\},\\
\cq_{nh}(n)&=\{R\in\cq(n)\mid R \text{ has non-trivial relations}\}.
\end{align*}
We say that the rooted quivers with relation in $\cq_h(n)$ are \emph{hereditary} and those in $\cq_{nh}(n)$ are \emph{non-hereditary}. 

\begin{lemma}
\label{lem:hereditary-QR}
\begin{itemize}
\item[(a)]
$\cq_h(n)$ consists of all rooted quivers (with trivial relation) whose underlying quiver is of type $\mathbb{A}_n$ and whose root has at most one neighbour. Moreover, $|\cq_h(n)|=2^{n-1}$.
\item[(b)] For $R\in\cq_h(n)$, if $n\geq 2$, then there are exactly two vertices which has one neighbour, one is the root, and we call the other the \emph{leaf} of $R$; if $n=1$, we  call the root the leaf. Switching the root and the leaf defines an involution on $\cq_h(n)\colon R\mapsto \bar{R}$. If two different elements $R,R'\in\cq_h(n)$ are isomorphic as quivers, then $R'=\bar{R}$.
\item[(c)] The number of $R\in\cq_h(n)$ with $R=\bar{R}$ is $2^{\frac{n-1}{2}}$ when $n$ is odd, and $0$ when $n$ is even.
\end{itemize}
\end{lemma}
\begin{proof}
(a) A rooted quiver whose underlying quiver is of type $\mathbb{A}_n$ and whose root has at most one neighbour must be of the form
\[\begin{xy}
(-12,0)*+{\bullet}="1",
(0,0)*+{\circ}="2",
(12,0)*+{\cdots}="3",
(28,0)*+{\circ}="4",
(42,0)*+{\circ}="5",
\ar@{-}"2";"1", \ar@{-}"3";"2",\ar@{-}"4";"3",\ar@{-}"5";"4",
\end{xy}
\]
where the $n-1$ arrows can be of any of the two orientations. Therefore the number of such rooted quivers is $2^{n-1}$. It remains to prove the first statement. It is trivial for $n=1$, because $\cq(1)=\{\bullet\}$. We assume $n\geq 2$ and let $R\in\cq_h(n)$. By Lemma~\ref{lem:vertices-lying-on-a-relation} (c), any vertex of $R$ has at most two neighbours (and at least one neighbour because $R$ is connected and has at least two vertices). This implies that $R$ as a quiver is of type $\mathbb{A}_n$. By Lemma~\ref{lem:vertices-lying-on-a-relation} (a), the root of $R$ is one of the two vertices of $R$ which has exactly one neighbour. Conversely, let $R$ be any rooted quiver whose underlying quiver is of type $\mathbb{A}_n$ and whose root has exactly one neighbour. We realise $R$ as a full connected rooted subquiver with relation of the genealogical tree as follows. First, the root of the genealogical tree is contained in $R$ as the root. Next, consider the arrow in $R$ between the root and its unique neighbour. If the root is its source, then $R$ contains the lower right neighbour of the root in the genealogical tree; if the root is its target, then $R$ contains the lower left neighbour of the root in the genealogical tree. Then we proceed with induction.

(b) Assume that $R$ and $R'$ are different elements of $\cq_h(n)$ and let $R\to R'$ be an isomorphism of quivers. Then it necessarily takes the root of $R$ to the leaf of $R'$ and takes the leaf of $R$ to the root of $R'$. It follows that $R'=\bar{R}$. 

(c) If $R=\bar{R}$, then switching the root and the leaf defines a non-trivial automorphism of the underlying quiver of $R$. Such an automorphism does not exist if $n$ is even. When $n$ is odd, the number of quivers of type $\mathbb{A}_n$ which admits such an automorphism is exactly $2^{\frac{n-1}{2}}$.
\end{proof}

\begin{lemma}
\label{lem:non-hereditary-QR}
Let $R\in\cq_{nh}(n)$ and $i,j$ be vertices of $R$. Assume that $i$ is not the root and lies on a relation $r$ of $R$. Let $p=\rho_1\cdots\rho_l$ be a walk in $R$ from $j$ to $i$ (where $\rho_1,\ldots,\rho_l$ are arrows or inverse arrows) and assume that $i$ is the first vertex on this walk which lies on a relation. 
\begin{itemize}
\item[(a)] If $j$ is the root, then $\rho_1$ is not involved in the relation $r$; if $j$ is not the root and is a source or a sink, then $\rho_1$ is involved in the relation $r$. 
\item[(b)] If $j$ is not the root, then the underlying quiver with relation of $R$ together with $j$ as the root does not belong to $\cq(n)$. 
\end{itemize}
\end{lemma}
\begin{proof}
(a) To go from $j$ to $i$ on the genealogical tree, if $j$ is the root, then all $\rho_l,\ldots,\rho_1$ are downward, and $\rho_1$ is not involved in the relation $r$; if $j$ is not the root and $j$ is a source or a sink, then all $\rho_l,\ldots,\rho_1$ are upward, and $\rho_1$ is involved in the relation $r$. 

(b) Let $R'$ be the rooted quiver with relation obtained from $R$ by changing the root to $j$ and suppose $R'\in\cq(n)$. Then by applying the first statement of (a) to $R'$ and applying the second statement of (a) to $R$, we see that $j$ cannot be a sink or a source of $R$. So by Lemma~\ref{lem:vertices-lying-on-a-relation} (a) applied to $R'$, $j$ is a vertex which lies on a relation and has two neighbours. But by Lemma~\ref{lem:vertices-lying-on-a-relation} (c) applied to $R$, $j$ has to be the root, a contradiction.
\end{proof}

\begin{corollary}
\label{cor:non-hereditary-QR}
If $R$ and $R'$ are different elements of $\cq_{nh}(n)$, then they are not isomorphic as quivers with relation.
\end{corollary}
\begin{proof}
By Lemma~\ref{lem:non-hereditary-QR} (b), any isomorphism $R'\to R$ of quivers with relation takes the root of $R'$ to the root of $R$. Then it follows that $R=R'$.
\end{proof}

\subsection{An operation}
\label{ss:an-operation}
For $R\in\cq(n)$, we call the full subquiver of $R$ lying on the left border of the genealogical tree the \emph{left border of $R$}. The left border of $R$ is necessarily a quiver of type $\mathbb{A}$ with linear orientation and with the unique sink as the root. Put
\begin{align*}
\cq^1(1)&=\cq(1), \text{~and~for~}n\geq 2\\
\cq^1(n)&=\{R\in\cq(n)\mid \text{the left border of $R$ has at least two vertices and there is exactly}\\
&\hspace{20pt}\text{one arrow going out from the unique source of the left border}\}.
\end{align*}
Assume that $n\geq 3$. For $R\in\cq^1(n)$, let $u$ denote the unique source of the left border of $R$, and let $v$ be the target of the unique arrow going out from $u$. Put
\begin{align*}
\cq^2(n)&=\{R\in\cq^1(n)\mid \text{$v$ lies on a relation of $R$}\},\\
\cq^3(n)&=\{R\in\cq^2(n)\mid \text{the rooted quiver with relation obtained from $R$}\\
&\hspace{20pt}\text{by removing the vertex $u$ is non-hereditary}\}. 
\end{align*}
Note that $\cq^3(3)$ and $\cq^3(4)$ are empty and $\cq^3(5)$ has two elements  (here we use a dotted curve to indicate the relation):
\begin{align}
\label{QR:Q(5)}
\begin{xy}
(0,-10) *+{u}="21",
(10,0) *+{\circ}="31",
(30,0) *+{\circ}="33",
(20,10) *+{\bullet}="32",
(16,6)*+{ }="34",
(24,6)*+{ }="35",
(20,-10)*+{\circ}="41",
(6,-4)*+{ }="44",
(14,-4)*+{ }="45",
\ar"31";"32", \ar"32";"33", \ar@/^0.6pc/@{.}"35";"34",
\ar"21";"31", \ar"31";"41", \ar@/^0.6pc/@{.}"45";"44",
\end{xy}\hspace{15pt},\hspace{15pt}
\begin{xy}
(10,0) *+{u}="31",
(30,0) *+{\circ}="33",
(20,10) *+{\bullet}="32",
(16,6)*+{ }="34",
(24,6)*+{ }="35",
(20,-10)*+{\circ}="41",
(26,-4)*+{ }="44",
(34,-4)*+{ }="45",
(40,-10)*+{\circ}="61",
\ar"31";"32", \ar"32";"33", \ar@/^0.6pc/@{.}"35";"34",
\ar@/^0.6pc/@{.}"45";"44",
\ar"33";"61", \ar"41";"33",
\end{xy}
\end{align}

For $R\in\cq^2(n)$, there must be a (unique) slanted downward arrow going out from $v$, which we denote by $\alpha$ and whose target we denote by $w$:
\[
\begin{xy}
 (-10,-4)*+{u}="1",
(0,6)*+{v}="2",
(10,-4)*+{w}="3",
(-4,2)*+{ }="4",
(4,2)*+{ }="5",
\ar"1";"2", \ar^\alpha"2";"3", \ar@/^0.6pc/@{.}"5";"4",
\end{xy}
\]
Define $\mu(R)$ as the rooted quiver with relation which is obtained from $R$ by deleting the vertex $u$ and adding
\begin{itemize}
\item[-] a vertex $x$,
\item[-] an arrow $\eta\colon w\to x$,
\item[-] a relation $\eta\alpha$:
\begin{align}
\label{QR:non-hereditary-A3}
\begin{xy}
(0,5) *+{v}="21",
(20,5) *+{x}="23",
(10,-5) *+{w}="22",
(6,-1)*+{ }="24",
(14,-1)*+{ }="25",
\ar_\alpha"21";"22", \ar_\eta"22";"23", \ar@/_0.6pc/@{.}"25";"24",
\end{xy}
\end{align}
\end{itemize}
If we apply $\mu$ to the two rooted quivers with relation in \eqref{QR:Q(5)}, we obtain
\begin{align}
\label{QR:M(5)}
\begin{xy}
(10,0) *+{\circ}="31",
(29,2) *+{\circ}="33",
(20,10) *+{\bullet}="32",
(29,-2) *+{x}="36",
(16,6)*+{ }="34",
(24,6)*+{ }="35",
(20,-10)*+{\circ}="41",
(16,-6)*+{ }="44",
(24,-6)*+{ }="45",
\ar"31";"32", \ar"32";"33", \ar@/^0.6pc/@{.}"35";"34",
\ar"31";"41", \ar@/_0.6pc/@{.}"45";"44",
\ar"41";"36",
\end{xy}\hspace{15pt},\hspace{15pt}
\begin{xy}
(40,10) *+{x}="31",
(30,0) *+{\circ}="33",
(20,10) *+{\bullet}="32",
(26,4)*+{ }="34",
(34,4)*+{ }="35",
(20,-10)*+{\circ}="41",
(26,-4)*+{ }="44",
(34,-4)*+{ }="45",
(40,-10)*+{\circ}="61",
\ar"33";"31", \ar"32";"33", \ar@/_0.6pc/@{.}"35";"34",
\ar@/^0.6pc/@{.}"45";"44",
\ar"33";"61", \ar"41";"33",
\end{xy}
\end{align}

Any $R\in\cq^2(n)$ belongs to one of the following three families:
\begin{itemize}
\item[(1)] the left border of $R$ has exactly two vertices (and the root of $\mu(R)$ is a source),
\item[(2)] the root of $R$ is a sink (and the root of $\mu(R)$ is a sink),
\item[(3)] the root of $R$ is not a sink and its left border has at least three vertices (and the root of $\mu(R)$ lies on a relation and has two neighbours).
\end{itemize}
Put 
\begin{align*}
\cm(n)&=\{\mu(R)\mid R\in\cq^3(n)\},\\
\widetilde{\cm}_k(n)&=\{\mu(R)\mid R\in\cq^2(n) \text{ belongs to the family (k)}\},~k=1,2,3,\\
\cm_k(n)&=\cm(n)\cap\widetilde{\cm}_k(n),~k=1,2,3.
\end{align*}
It is clear that elements of $\cm(n)$ are pairwise non-isomorphic as rooted quivers with relation, $\cm(n)=\cm_1(n)\sqcup \cm_2(n)\sqcup \cm_3(n)$, and that $\mu\colon \cq^3(n)\to \cm(n)$ is a bijection. Notice also that $\cm(3)$ and $\cm(4)$ are empty, $\cm_2(5)$ is also empty and both $\cm_1(5)$ and $\cm_3(5)$ have one element, see \eqref{QR:M(5)}.

\smallskip
For $R_1\in\cq(n-2)$, we define $E(R_1)$ as the rooted quiver with relation obtained from $R_1$  and the quiver with relation
\[
\begin{xy}
(0,0) *+{1}="21",
(20,0) *+{3}="23",
(10,-10) *+{2}="22",
(6,-6)*+{ }="24",
(14,-6)*+{ }="25",
\ar_{\eta_1}"21";"22", \ar_{\eta_2}"22";"23", \ar@/_0.6pc/@{.}"25";"24",
\end{xy}
\]
by identifying the root of $R_1$ with the vertex $2$. The vertex $1$ is the root of $E(R_1)$. We will also consider $E(R_1)$ as a quiver with relation with distinguished vertices $1,2,3$. For example, 
\[
E(\begin{xy}
 (-7.5,-3)*+{\circ}="1",
(0,4.5)*+{\bullet}="2",
\ar"1";"2",
\end{xy})=
\begin{xy}
 (-10,-10)*+{\circ}="1",
(0,0)*+{2}="2",
(-10,10) *+{1}="21",
(10,10) *+{3}="23",
(-4,4)*+{ }="24",
(4,4)*+{ }="25",
\ar"21";"2", \ar"2";"23", \ar@/_0.6pc/@{.}"25";"24",
\ar"1";"2", 
\end{xy},\hspace{15pt}
E(\begin{xy}
 (-7.5,-3)*+{\circ}="1",
(0,4.5)*+{\bullet}="2",
(7.5,-3)*+{\circ}="3",
(-3,1.5)*+{ }="4",
(3,1.5)*+{ }="5",
\ar"1";"2", \ar"2";"3", \ar@/^0.6pc/@{.}"5";"4",
\end{xy})=
\begin{xy}
 (-10,-10)*+{\circ}="1",
(0,0)*+{2}="2",
(10,-10)*+{\circ}="3",
(-4,-4)*+{ }="4",
(4,-4)*+{ }="5",
(-10,10) *+{1}="21",
(10,10) *+{3}="23",
(-4,4)*+{ }="24",
(4,4)*+{ }="25",
\ar"21";"2", \ar"2";"23", \ar@/_0.6pc/@{.}"25";"24",
\ar"1";"2", \ar"2";"3", \ar@/^0.6pc/@{.}"5";"4",
\end{xy}.
\]

\begin{lemma}
\label{lem:non-hereditary-mutated-QR-1}
\begin{itemize}
\item[(a)] For $R_1\in\cq(n-2)$, $E(R_1)$ belongs to $\cq(n)$ if and only if $R_1$ is hereditary.
\item[(b)] We have
\begin{align*}
\widetilde{\cm}_1(n)&=\{E(R_1)\mid R_1\in\cq(n-2)\},\\
\cm_1(n)&=\{E(R_1)\mid R_1\in\cq(n-2) \text{ is non-hereditary}\}.
\end{align*}
\end{itemize}
\end{lemma}
\begin{proof}
(a) If $R_1$ is hereditary, then by turning $E(R_1)$ upside down, we see that it belongs to $\cq(n)$ with the leaf of $R_1$ as the root. If $R_1$ is non-hereditary, then $E(R_1)$ has a full subquiver with relation of the form
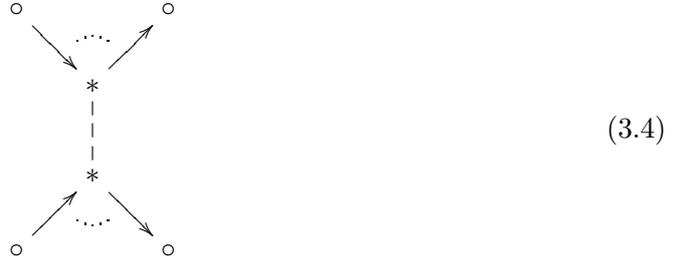
\begin{align}
\label{QR:type-X}
\begin{xy}
(0,16) *+{\circ}="1",
(20,16) *+{\circ}="3",
(10,6) *+{*}="2",
(6,10)*+{ }="24",
(14,10)*+{ }="25",
(0,-16) *+{\circ}="4",
(10,-6) *+{*}="5",
(20,-16) *+{\circ}="6",
(6,-10)*+{ }="34",
(14,-10)*+{ }="35",
\ar"1";"2", \ar"2";"3", \ar@{--}"2";"5", \ar@/_0.6pc/@{.}"25";"24",
\ar"4";"5", \ar"5";"6", \ar@/^0.6pc/@{.}"35";"34",
\end{xy}
\end{align}
where the dashed line between the two $*$'s is a quiver of type $\mathbb{A}_{m-4}$ for some $m$. This cannot belong to $\cq(m)$, because no vertex can serve as the root by Lemma~\ref{lem:vertices-lying-on-a-relation} (a) and Lemma~\ref{lem:non-hereditary-QR} (a). Thus by Lemma~\ref{lem:vertices-lying-on-a-relation} (b), $E(R_1)$ does not belong to $\cq(n)$.

(b) Let $R\in\cq^2(n)$ belong to the family (1). Then the vertex $v$ of $\mu(R)$ is the root and a source of $\mu(R)$. We double the vertex $w$ and cut $\mu(R)$ into two at $w$ to obtain two quivers with relation: one consists of the vertices $v$, $w$ and $x$ (of the form \eqref{QR:non-hereditary-A3}), and the other we denote by $R_1$. We consider $R_1$ as a rooted quiver with relation with root $w$. It can also be obtained from $R$ by removing the vertices $u$ and $v$, so it belongs to $\cq(n-2)$. Then $\mu(R)=E(R_1)$. This shows the inclusion `$\subseteq$' in the first equality. For the other direction, let $R_1\in\cq(n-2)$ and we perform the following operation to $E(R_1)$ to obtain a rooted quiver with relation $R$: 
\begin{itemize}
\item[-] remove the vertex $3$,
\item[-] add a vertex $0$,
\item[-] add an arrow $\delta\colon 0\to 1$,
\item[-] add a relation $\eta_1\delta$.
\end{itemize}
Then $R\in\cq^2(n)$ belongs to the family (1), and $E(R_1)=\mu(R)\in\widetilde{\cm}_1(n)$.

Moreover, $\mu(R)=E(R_1)$ belongs to $\cm_1(n)$ if and only if the quiver with relation obtained from $R$ by removing $u$ is non-hereditary, that is, the quiver with relation obtained from $E(R_1)$ by removing the vertex $3$ is non-hereditary, which happens if and only if $R_1$ is non-hereditary. This shows the second equality.
\end{proof}

Let $2\leq j\leq n-2$. For $R_1\in\cq^1(j-1)$ and $R_2\in\cq(n-j-1)$, define $E(R_1,R_2)$ as the rooted quiver with relation obtained from $R_1$ and $R_2$ and the quiver with relation
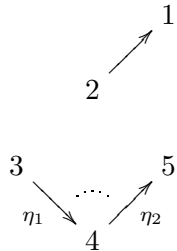
\[
\begin{xy}
(0,0) *+{3}="21",
(20,0) *+{5}="23",
(10,-10) *+{4}="22",
(6,-6)*+{ }="24",
(14,-6)*+{ }="25",
(10,10) *+{2}="26",
(20,20) *+{1}="27",
\ar_{\eta_1}"21";"22", \ar_{\eta_2}"22";"23", \ar@/_0.6pc/@{.}"25";"24",
\ar"26";"27",
\end{xy}
\]
by identifying the root of $R_1$ with the vertex $2$, the source of the left border of $R_1$ with the vertex $3$, and the root of $R_2$ with the vertex $4$. The vertex $1$ is the root of $E(R_1,R_2)$. We will also consider $E(R_1,R_2)$ as a quiver with relation with distinguished vertices $1,2,3,4,5$. For example, 
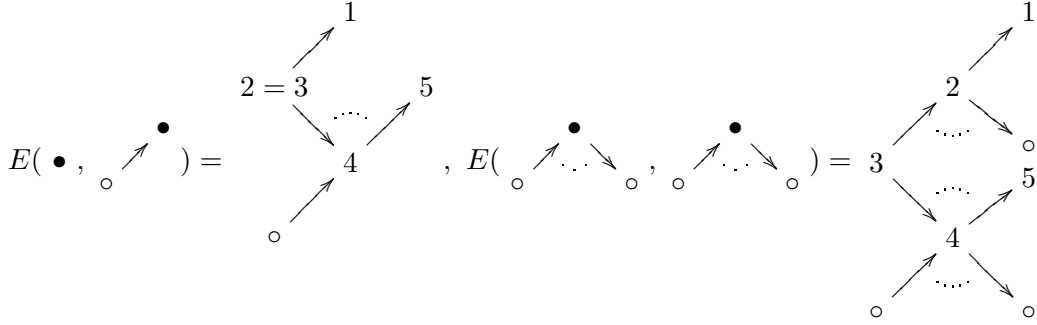
\[
E(~\bullet~, \begin{xy}
 (-7.5,-3)*+{\circ}="1",
(0,4.5)*+{\bullet}="2",
\ar"1";"2", 
\end{xy})=
\begin{xy}
 (-10,-10)*+{\circ}="1",
(0,0)*+{4}="2",
(-10,10) *+{2=3}="21",
(10,10) *+{5}="23",
(-4,4)*+{ }="24",
(4,4)*+{ }="25",
(0,20) *+{1}="27",
\ar"21";"2", \ar"2";"23", \ar@/_0.6pc/@{.}"25";"24",
\ar"1";"2", \ar"21";"27",
\end{xy}
,~~
E(\begin{xy}
 (-7.5,-3)*+{\circ}="1",
(0,4.5)*+{\bullet}="2",
(7.5,-3)*+{\circ}="3",
(-3,1.5)*+{ }="4",
(3,1.5)*+{ }="5",
\ar"1";"2", \ar"2";"3", \ar@/^0.6pc/@{.}"5";"4",
\end{xy}
,\begin{xy}
 (-7.5,-3)*+{\circ}="1",
(0,4.5)*+{\bullet}="2",
(7.5,-3)*+{\circ}="3",
(-3,1.5)*+{ }="4",
(3,1.5)*+{ }="5",
\ar"1";"2", \ar"2";"3", \ar@/^0.6pc/@{.}"5";"4",
\end{xy})=
\begin{xy}
 (-10,-20)*+{\circ}="1",
(0,-10)*+{4}="2",
(10,-20) *+{\circ}="3",
(-4,-14)*+{ }="4",
(4,-14)*+{ }="5",
(-10,0) *+{3}="21",
(10,-2) *+{5}="23",
(-4,-6)*+{ }="24",
(4,-6)*+{ }="25",
(10,20) *+{1}="27",
(0,10)*+{2}="32",
(-4,6)*+{ }="34",
(4,6)*+{ }="35",
(10,2) *+{\circ}="36",
\ar"2";"3", \ar@/^0.6pc/@{.}"5";"4",
\ar"21";"2", \ar"2";"23", \ar@/_0.6pc/@{.}"25";"24",
\ar"1";"2", 
\ar"21";"32", \ar@/^0.6pc/@{.}"35";"34",
\ar"32";"27", \ar"32";"36",
\end{xy}
\]
\begin{lemma}
\label{lem:non-hereditary-mutated-QR-2}
\begin{itemize}
\item[(a)] For $2\leq j\leq n-2$, $R_1\in\cq^1(j-1)$ and $R_2\in\cq(n-j-1)$, $E(R_1,R_2)$ belongs to $\cq(n)$ if and only if both $R_1$ and $R_2$ are hereditary.
\item[(b)] We have
\begin{align*}
\widetilde{\cm}_2(n)&=\{E(R_1,R_2)\mid R_1\in\cq^1(j-1),~R_2\in\cq(n-j-1), 2\leq j\leq n-2\},\\
\cm_2(n)&=\{E(R_1,R_2)\mid R_1\in\cq^1(j-1),~R_2\in\cq(n-j-1),~\text{and at least one}\\
&\hspace{20pt} \text{of $R_1$ and $R_2$ is non-hereditary}, 2\leq j\leq n-2\},\\
\end{align*}
\end{itemize}
\end{lemma}
\begin{proof}
(a) If both $R_1$ and $R_2$ are hereditary, then by turning $E(R_1,R_2)$ upside down, we see that it belongs to $\cq(n)$ with the leaf of $R_2$ as the root. If $R_1$ is non-hereditary, then $E(R_1,R_2)$ has a full subquiver with relation of the form
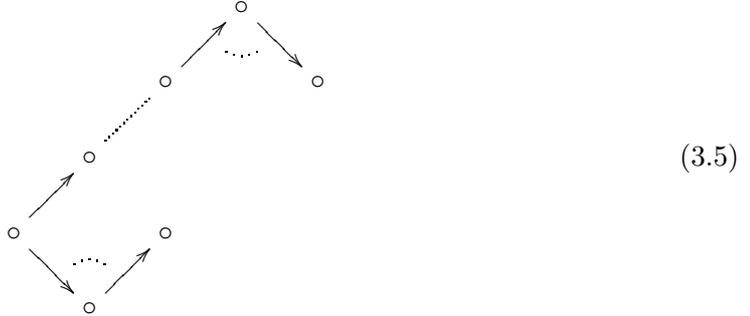
\begin{align}
\label{QR:type-box}
\begin{xy}
(0,-10) *+{\circ}="21",
(20,-10) *+{\circ}="23",
(10,-20) *+{\circ}="22",
(6,-16)*+{ }="24",
(14,-16)*+{ }="25",
(10,0) *+{\circ}="26",
(20,10) *+{\circ}="27",
(30,20) *+{\circ}="28",
(40,10) *+{\circ}="29",
(26,16)*+{ }="34",
(34,16)*+{ }="35",
\ar"21";"22", \ar"22";"23", \ar@/_0.6pc/@{.}"25";"24",
\ar@{.}"26";"27", \ar"21";"26", \ar"27";"28", \ar"28";"29",
 \ar@/^0.6pc/@{.}"35";"34",
\end{xy}
\end{align}
which, by Lemma~\ref{lem:vertices-lying-on-a-relation} (c), does not belong to $\cq(m)$ for any $m$, because otherwise there are two roots. Thus by Lemma~\ref{lem:vertices-lying-on-a-relation} (b), $E(R_1,R_2)$ does not belong to $\cq(n)$. If $R_2$ is non-hereditary, then $E(R_1,R_2)$ has a full subquiver with relation of the form \eqref{QR:type-X}, so it does not belong to $\cq(n)$.

(b) Let $R\in\cq^2(n)$ belong to the family (2). Let $y$ be the vertex from which there is an arrow to the root. We double $y$, $v$ and $w$ and cut $\mu(R)$ at $y$, $v$ and $w$ to obtain four quivers with relation: the one consisting of the root and $y$ (of type $\mathbb{A}_2$), the one consisting of $v$, $w$ and $x$  (of the form \eqref{QR:non-hereditary-A3}), and for the other two, we denote the upper one by $R_1$ and the lower one by $R_2$. Then $R_1\in\cq^1(j-1)$ and $R_2\in\cq(n-j-1)$, and $\mu(R)=E(R_1,R_2)$. This shows the inclusion `$\subseteq$' in the first equality. For the other direction, let $R_1\in\cq^1(j-1)$ and $R_2\in\cq(n-j-1)$ and we perform the following operation to $E(R_1,R_2)$ to obtain a rooted quiver with relation $R$: 
\begin{itemize}
\item[-] remove the vertex $5$,
\item[-] add a vertex $0$,
\item[-] add an arrow $\delta\colon 0\to 3$,
\item[-] add a relation $\eta_1\delta$.
\end{itemize}
Then $R\in\cq^2(n)$ belongs to the family (2), and $E(R_1,R_2)=\mu(R)\in\widetilde{\cm}_2(n)$.

Moreover, $\mu(R)=E(R_1,R_2)$ belongs to $\cm_2(n)$ if and only if the quiver with relation obtained from $R$ by removing $u$ is non-hereditary, that is, the quiver with relation obtained from $E(R_1,R_2)$ by removing the vertex $5$ is non-hereditary, which happens if and only if $R_1$ is non-hereditary or $R_2$ is non-hereditary. This shows the second equality.
\end{proof}

Let $3\leq i\leq j\leq n-2$. For $R_1\in\cq^1(j-i+1)$, $R_2\in\cq(n-j-1)$ and $R_3\in\cq(i-2)$, define $E(R_1,R_2,R_3)$ as the rooted quiver with relation obtained from $R_1$, $R_2$, $R_3$ and the quiver with relation
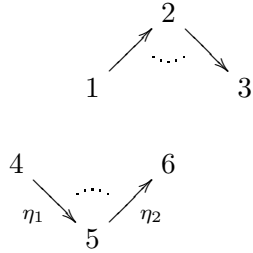
\[
\begin{xy}
(0,0) *+{4}="21",
(20,0) *+{6}="23",
(10,-10) *+{5}="22",
(6,-6)*+{ }="24",
(14,-6)*+{ }="25",
(10,10) *+{1}="31",
(30,10) *+{3}="33",
(20,20) *+{2}="32",
(16,16)*+{ }="34",
(24,16)*+{ }="35",
\ar_{\eta_1}"21";"22", \ar_{\eta_2}"22";"23", \ar@/_0.6pc/@{.}"25";"24",
\ar"31";"32", \ar"32";"33", \ar@/^0.6pc/@{.}"35";"34",
\end{xy}
\]
by identifying the root of $R_1$ with the vertex $1$, the source of the left border of $R_1$ with the vertex $4$, the root of $R_2$ with the vertex $5$, and the root of $R_3$ with the vertex $3$. The vertex $2$ is the root of $E(R_1,R_2,R_3)$. We will also consider $E(R_1,R_2,R_3)$ as a quiver with relation with distinguished vertices $1,2,3,4,5,6$. For example,
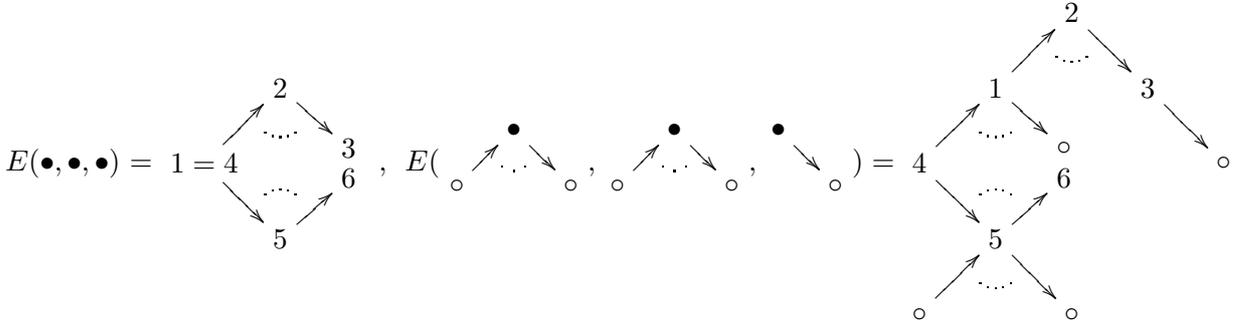
\[
E(\bullet,\bullet,\bullet)=\begin{xy}
(10,0) *+{1=4}="31",
(29,2) *+{3}="33",
(20,10) *+{2}="32",
(29,-2) *+{6}="36",
(16,6)*+{ }="34",
(24,6)*+{ }="35",
(20,-10)*+{5}="41",
(16,-6)*+{ }="44",
(24,-6)*+{ }="45",
\ar"31";"32", \ar"32";"33", \ar@/^0.6pc/@{.}"35";"34",
\ar"31";"41", \ar@/_0.6pc/@{.}"45";"44",
\ar"41";"36",
\end{xy}\hspace{5pt},\hspace{5pt}
E(\begin{xy}
 (-7.5,-3)*+{\circ}="1",
(0,4.5)*+{\bullet}="2",
(7.5,-3)*+{\circ}="3",
(-3,1.5)*+{ }="4",
(3,1.5)*+{ }="5",
\ar"1";"2", \ar"2";"3", \ar@/^0.6pc/@{.}"5";"4",
\end{xy},
\begin{xy}
(-7.5,-3)*+{\circ}="1",
(0,4.5)*+{\bullet}="2",
(7.5,-3)*+{\circ}="3",
(-3,1.5)*+{ }="4",
(3,1.5)*+{ }="5",
\ar"1";"2", \ar"2";"3", \ar@/^0.6pc/@{.}"5";"4",
\end{xy}
,\begin{xy}
 (-7.5,4.5)*+{\bullet}="1",
(0,-3)*+{\circ}="2",
\ar"1";"2", 
\end{xy})=
\begin{xy}
(0,0) *+{4}="21",
(19,-2) *+{6}="23",
(10,-10) *+{5}="22",
(6,-6)*+{ }="24",
(14,-6)*+{ }="25",
(10,10) *+{1}="31",
(30,10) *+{3}="33",
(20,20) *+{2}="32",
(16,16)*+{ }="34",
(24,16)*+{ }="35",
(19,2)*+{\circ}="41",
(6,6)*+{ }="44",
(14,6)*+{ }="45",
(0,-20) *+{\circ}="51",
(20,-20) *+{\circ}="53",
(6,-14)*+{}="54",
(14,-14)*+{ }="55",
(40,0)*+{\circ}="61",
\ar"21";"22", \ar"22";"23", \ar@/_0.6pc/@{.}"25";"24",
\ar"31";"32", \ar"32";"33", \ar@/^0.6pc/@{.}"35";"34",
\ar"21";"31", \ar"31";"41", \ar@/^0.6pc/@{.}"45";"44",
\ar"51";"22", \ar"22";"53", \ar@/^0.6pc/@{.}"55";"54",
\ar"33";"61",
\end{xy}
\]

\begin{lemma}
\label{lem:non-hereditary-mutated-QR-3}
\begin{itemize}
\item[(a)]
For any $3\leq i\leq j\leq n-2$ and any $R_1\in\cq^1(j-i+1)$, $R_2\in\cq(n-j-1)$ and $R_3\in\cq(i-2)$, $E(R_1,R_2,R_3)$ does not belong to $\cq(n)$.
\item[(b)] We have
\begin{align*}
\widetilde{\cm}_3(n)=\cm_3(n)&=\{E(R_1,R_2,R_3)\mid R_1\in \cq^1(j-i+1),~R_2\in\cq(n-j-1),\\
&\hspace{15pt}\text{and}~R_3\in\cq(i-2), 3\leq i\leq j\leq n-2\}.
\end{align*}
\end{itemize}
\end{lemma}
\begin{proof}
(a) $E(R_1,R_2,R_3)$ always has a full subquiver with relation of the form \eqref{QR:type-box}, so it does not belong to $\cq(n)$.

(b) Let $R\in\cq^2(n)$ belong to the family (3). Let $y$ be the vertex from which there is an arrow to the root and $z$ be the vertex to which there is an arrow from the root. We double $y$, $z$, $v$ and $w$ and cut $\mu(R)$ at $y$, $z$, $v$ and $w$ to obtain five quivers with relation: the one consisting of the root, $y$ and $z$, the one consisting of $v$, $w$ and $x$, and for the other three, we denote the left upper one by $R_1$, the lower one by $R_2$ and the right upper one by $R_3$. Then $R_1\in\cq^1(j-i+1)$, $R_2\in\cq(n-j-1)$ and $R_3\in\cq(i-2)$ for some $3\leq i\leq j\leq n-2$, and $\mu(R)=E(R_1,R_2,R_3)$. This shows the inclusion `$\subseteq$' in the first equality. For the other direction, let $R_1\in\cq^1(j-i+1)$, $R_2\in\cq(n-j-1)$ and $R_3\in\cq(i-2)$ and we perform the following operation to $E(R_1,R_2,R_3)$ to obtain a rooted quiver with relation $R$: 
\begin{itemize}
\item[-] remove the vertex $6$,
\item[-] add a vertex $0$,
\item[-] add an arrow $\delta\colon 0\to 4$,
\item[-] add a relation $\eta_1\delta$.
\end{itemize}
Then $R\in\cq^2(n)$ belongs to the family (3), and $E(R_1,R_2,R_3)=\mu(R)\in\widetilde{\cm}_3(n)$.
\end{proof}

\section{Silted algebras of type $\Lambda_n$}
\label{s:silted-algebra-linear-orientation}

In this section we classify, up to isomorphism, all basic silted algebras of type $\Lambda_n$, the path algebra of the quiver
\begin{center}
$\begin{xy}
(-20,0)*+{\overrightarrow{\mathbb{A}_{n}}\colon}="0",
(-12,0)*+{1}="1",
(0,0)*+{2}="2",
(12,0)*+{\cdots}="3",
(28,0)*+{n-1}="4",
(42,0)*+{n}="5",
\ar"2";"1", \ar"3";"2",\ar"4";"3",\ar"5";"4",
\end{xy}$
\end{center}
Moreover, we calculate the number of these silted algebras. 

\subsection{Main result}
Put
\begin{align*}
\ca_t(\Lambda_n)&:=\{\text{basic tilted algebras of type $\Lambda_n$}\}/\cong~,\\
\ca_s(\Lambda_n)&:=\{\text{basic silted algebras of type $\Lambda_n$}\}/\cong~,\\
\ca_{ns}(\Lambda_n)&:=\{\text{basic non-connected silted algebras of type $\Lambda_n$}\}/\cong~.
\end{align*}
Let $a_t(\Lambda_n)$, $a_{ns}(\Lambda_n)$ and $a_s(\Lambda_n)$ denote the cardinalities of $\ca_t(\Lambda_n)$, $\ca_{ns}(\Lambda_n)$ and $\ca_s(\Lambda_n)$, respectively.
The main result of this section is:

\begin{theorem}
\label{thm:the-number-of-silted-algebras-of-Lambda_n}
\begin{itemize}
\item [(a)] 
$\ca_s(\Lambda_n)=\ca_t(\Lambda_n)\sqcup \ca_{ns}(\Lambda_n),$
where
\[
\ca_{ns}(\Lambda_n)=\bigsqcup_{m=1}^{[\frac{n}{2}]} (\ca_t(\Lambda_m)\times_s \ca_t(\Lambda_{n-m}))
\]
\item [(b)] $a_{s}(\Lambda_n)=a_{t}(\Lambda_n) + a_{ns}(\Lambda_n)$, 
where
\begin{align*}
a_{t}(\Lambda_{n})&=\frac{1}{n+1}(\begin{smallmatrix}2n\\n\end{smallmatrix})+(1-(-1)^{n})\times 2^{[\frac{n}{2}]-2}-2^{n-2},\\
a_{ns}(\Lambda_n)&=\begin{cases}
\sum\limits_{m=1}^{\frac{n-1}{2}} a_{t}(\Lambda_{n-m}) \times a_{t}(\Lambda_{m}), &\text{  if $n$ is odd},\\
\sum\limits_{m=1}^{\frac{n}{2}-1} a_{t}(\Lambda_{n-m}) \times a_{t}(\Lambda_{m})+\frac{a_{t}(\Lambda_{n/2})\times (a_{t}(\Lambda_{n/2})+1)}{2}, &\text{ if $n$ is even},
\end{cases}
\end{align*}
\end{itemize}
\end{theorem}
This theorem will be proved in Section~\ref{ss:silted-algebras-type-Lambdan}. Note that as a consequence of (a), there are no strictly shod algebras in $\ca_s(\Lambda_n)$. 
The following table contains the first values of $a_t(\Lambda_n)$, $a_{ns}(\Lambda)$ and $a_s(\Lambda_n)$.

$$\begin{tabular}{|c||c|c|c|c|c|c|c|c|c|c| p{<20>}}
\hline
$n$ & $1$ & $2$ & $3$ & $4$ & $5$ & $6$ & $7$ & $8$ & $9$ & $10$\\
\hline
$a_t(\Lambda)$ & 1 & 1 & 4 & 10 & 36 & 116 & 401 & 1366 & 4742 & 16540 \\
\hline
$a_{ns}(\Lambda)$ & 0 & 1  & 1 & 5 & 14  & 56  & 192 & 716 & 2591 & 9538 \\
\hline
$a_s(\Lambda_n)$ & 1 & 2  & 5 & 15 & 50  & 172  & 593 & 2082 & 7333 & 26078\\
\hline
\end{tabular} $$

\subsection{Tilting modules over $\Lambda_n$}

In this subsection, we recall some facts on tilting modules over $\Lambda_n$. We first draw the Auslander--Reiten quiver of $\mod \Lambda_n$ in Figure~1. 
Here for a vertex $i$ of the quiver, we denote by $P(i)$, $I(i)$ and $S(i)$ the indecomposable projective, indecomposable injective and simple module at $i$, respectively. Note that $P(1)=I(n)$.
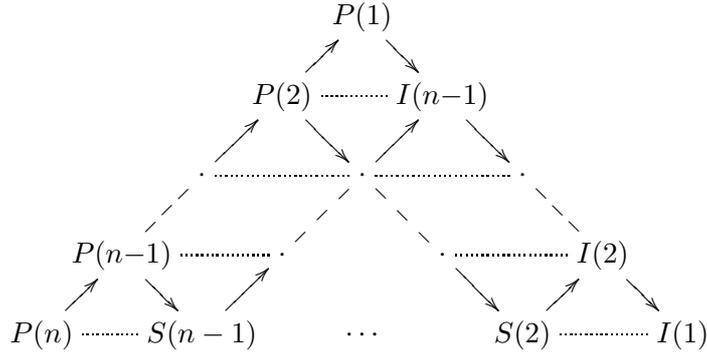
\begin{figure}
\label{fig:AR-quiver-of-Lambdan}
$\begin{xy}
0;<3pt,0pt>:<0pt,3pt>::
 (0,0)*+{P(1)}="1",
(-10,-10)*+{P(2)}="2",
(10,-10)*+{I(n{-}1)}="3",
(-20,-20)*+{\cdot}="4",
(0,-20)*+{\cdot}="5",
(20,-20)*+{\cdot}="6",
(-30,-30)*+{P(n{-}1)}="7",
(-10,-30)*+{\cdot}="8",
(10,-30)*+{\cdot}="9",
(30,-30)*+{I(2)}="10",
(-40,-40)*+{P(n)}="11",
(-20,-40)*+{S(n-1)}="12",
(0,-40)*+{\cdots}="13",
(20,-40)*+{S(2)}="14",
(40,-40)*+{I(1)}="15",
\ar"2";"1", \ar"2";"5", \ar@{.}"2";"3", \ar"5";"3", \ar"1";"3", \ar"4";"2", \ar@{.}"4";"5", \ar"3";"6", \ar@{.}"5";"6",
\ar@{--}"7";"4", \ar@{.}"7";"8", \ar@{--}"6";"10", \ar@{--}"5";"9", \ar"9";"14", \ar@{.}"9";"10", \ar@{--}"8";"5",
\ar"11";"7", \ar"7";"12", \ar"10";"15",\ar"14";"10", \ar"12";"8", \ar@{.}"11";"12",  \ar@{.}"14";"15",
\end{xy}$
\caption{The~{Auslander--Reiten}~quiver~of~$\mod \Lambda_n$}
\end{figure}

In this Auslander--Reiten quiver, the \emph{upper ray} (respectively, \emph{lower ray}) of an indecomposable module $M$ consists of $M$ and the indecomposable modules which can be reached from $M$ through upper (respectively, lower) slanted arrows only, and the \emph{hammock} starting at $M$ consists of indecomposable modules which lies in the rectangle expanded by the upper ray and the lower ray of $M$. For example, the hammock starting at a simple module $S(i)$ is exactly the upper ray of $S(i)$. These hammocks are described in the proof of Lemma (4.1) in \cite{HappelRingel81}. For the general theory on hammocks see \cite{RingelVossieck87}.

Notice that any tilting module over $\Lambda_n$ must contain $P(1)$ as a direct summand, because it is a projective-injective module. 
The following useful fact is a consequence of the first part of the proof of Lemma (4.1) in \cite{HappelRingel81}.

\begin{lemma}\label{lem:direct-summand-and-1-orbit}
\label{lem:direct-summand-and-j-orbit}
\label{lem:Ext-and-hammock}
Let $T$ be a basic tilting module over $\Lambda_n$ and $2\leq i\leq n$. Then $T$ has $P(i)$ as a direct summand if and only if $T$ has no direct summands in the hammock starting at $\tau^{-1}P(i)$.
\end{lemma}

%
%

\begin{remark}\label{rek:tilting-module-over-A-turn-to-B}
Let $T$ be a basic tilting module over $\Lambda_n$ containing $P(n)$ as a direct summand. By Lemma \ref{lem:direct-summand-and-1-orbit}, $T$ has no direct summand in the upper ray of $S(n-1)$. Write $T=P(n)\oplus T'$. Then $T^{\prime}$ can be viewed as a basic tilting module in the additive closure of all indecomposable modules which is neither $P(n)$ nor in the upper ray of $S(n-1)$. This additive closure is equivalent to $\mod \Lambda_{n-1}$.
\end{remark}

The following result is a consequence of Lemma~\ref{lem:direct-summand-and-j-orbit}. Inductively applying it, we can obtain a classification of basic tilting modules over $\Lambda_n$. Recall from \cite[3.3]{Ringel84} that for an indecomposable module $M$ over $\Lambda_n$, a full translation subquiver $\Gamma'$ of the AR-quiver $\Gamma$ of $\mod \Lambda_n$ is called the \emph{wing} of $M$, if for  $z$ of $\Gamma'$, all direct predecessors of $z$ in $\Gamma$ belong to $\Gamma'$, $\Gamma'$ is of the form of the AR-quiver of $\mod \Lambda_{m}$ for some $m$ and $M$ is the projective-injective vertex of $\Gamma'$. The following proposition is well-known.

\begin{proposition}
\label{prop:tilting-modules-of-Lambdan}
Let $T$ be a basic tilting module over $\Lambda_n$. Then $T$ is of one of the following two forms:
\begin{itemize}
\item[(1)] $T=P(1)\oplus T'$, where $T'$ is a basic tilting module of the wing of $P(2)$ or a basic tilting module of the wing of  $I(n{-}1)$;

\item[(2)] $T=P(1)\oplus T'\oplus T''$, where $T'$ is a basic tilting module of the wing of $P(i)$  for some $3\leq i\leq n$ and $T''$ is a basic tilting module of the wing of $I(i-2)$.
\end{itemize}
\end{proposition}

\begin{remark} 
\label{rmk:tilting-modules-of-Lambdan}
Let $T$ be a basic tilting module over $\Lambda_n$ and $M$ be an indecomposable direct summand of $T$. Consider the indecomposable direct summands of $T$ which belong to the wing $\mathcal{W}$ of $M$. The direct sum of them must be a basic tilting module in $\mathcal{W}$, by repeatedly applying Proposition~\ref{prop:tilting-modules-of-Lambdan}.
\end{remark}

\begin{lemma} \label{lem:maximal-P(i)-to-relation}
Let $T$ be a basic tilting module over $\Lambda_n$ which contains $P(n)$ but not $P(n{-}1)$ as a direct summand. Assume that $T$ contains $P(i)$ as a direct summand for some $1\leq i\leq n{-}2$, and let $j$ be the maximal such $i$. Then $T$ must contain $\tau^{-2}P(j+2)$ as a direct summand but not $\tau^{-1}P(j+1)$.
\end{lemma}

\begin{proof}
Consider the wing $\mathcal{W}$ of $P(j)$. By Remark~\ref{rmk:tilting-modules-of-Lambdan}, $T$ contains as a direct summand a basic tilting module $T'$ in $\mathcal{W}$. Since $P(j)$ is the projective-injective object in $\mathcal{W}$ and all other projective objects except $P(n)$ are not direct summands of $T'$, it follows by Proposition~\ref{prop:tilting-modules-of-Lambdan}(2) applied to $\mathcal{W}$ that $T'$, and hence $T$, must contain $\tau^{-2}P(j+2)$ as a direct summand but not $\tau^{-1}P(j+1)$.
\end{proof}

Finally, we have the following formula for the number $t(\Lambda_n)$ of isoclasses of basic tilting modules over $\Lambda_n$.
\begin{proposition}[{\cite[Theorem 1]{ObaidNaumanFakiehRingel14}}]
\label{prop:number-of-tilting-modules-over-An}
$
t(\Lambda_n)=\frac{1}{n+1}(\begin{smallmatrix}2n\\n\end{smallmatrix}).
$
\end{proposition}

\subsection{Tilted algebras of type $\Lambda_n$}
In this subsection, we study tilted algebras of type $\Lambda_n$.  
Let $\ct(\Lambda_n)$ be the set of isoclasses of basic tilting modules over $\Lambda_n$, and let $\varepsilon\colon\ct(\Lambda_n)\to\ca_t(\Lambda_n)$ be the map of taking the endomorphism algebra.

To each $T\in\ct(\Lambda_n)$ we associate a rooted quiver with relation $\varepsilon_1(T)\in\cq(n)$ as follows. First, $\varepsilon_1(T)$ contains the root of the genealogical tree (which corresponds to the direct summand $P(1)$ of T). Next, if $T=P(1)\oplus T'$ with $T'$ a basic tilting module in the wing of $P(2)$, then $\varepsilon_1(T)$ contains the left lower neighbour of the root; if $T=P(1)\oplus T'$ with $T'$ a basic tilting module in the wing of $I(n-1)$, then $\varepsilon_1(T)$ contains the right lower neighbour of the root; otherwise, $\varepsilon_1(T)$ contains both the left lower neighbour and the right lower neighbour of the root. Then we proceed by induction, thanks to Propostion~\ref{prop:tilting-modules-of-Lambdan}. Part (a) of the following lemma is a reformulation of the main result of the second part of \cite[(4.1)]{HappelRingel81}. Part (b) follows from Lemma~\ref{lem:hereditary-QR} (b) and Corollary~\ref{cor:non-hereditary-QR}. Part (c) follows from the construction of $\varepsilon_1(T)$.

\begin{lemma}\label{lem:genealogical-tree}
\begin{itemize}
\item[(a)] The map $\varepsilon_1\colon\ct(\Lambda_n)\to\cq(n)$ is bijective. Moreover, $\varepsilon$ is the composition of $\varepsilon_1$ with the map of taking a rooted quiver with relation to its path algebra.
Consequently, a tilted algebra of type $\Lambda_n$ must be the path algebra of a rooted quiver with relation in $\cq(n)$.
\item[(b)]  
For any $B\in\ca_t(\Lambda_n)$, the fibre $\varepsilon^{-1}(B)$ consists of either one element or two elements, and $\varepsilon^{-1}(B)$ consists of one element if $B$ is non-hereditary.
\item[(c)] Let $T\in\ct(\Lambda_n)$. If $T=P(1)\oplus T'$ with $T'$ a basic tilting module in the wing of $P(2)$, then the root is a sink of $\varepsilon_1(T)$; if $T=P(1)\oplus T'$ with $T'$ a basic tilting module in the wing of $I(n-1)$, then the root is a source of $\varepsilon_1(T)$; otherwise, the root of $\varepsilon_1(T)$ lies on a relation and has two neighbours.
\end{itemize}
\end{lemma}

Let $a_{ht}(\Lambda_n)$ be the number of isoclasses of basic hereditary tilted algebras of type $\Lambda_n$.
It follows from Lemma~\ref{lem:hereditary-QR} and Lemma~\ref{lem:genealogical-tree} that the path algebra of any quiver of type $\mathbb{A}_n$ is tilted of type $\Lambda_n$ and they are exactly all the basic hereditary tilted algebra of type $\Lambda_n$.  For $n=1,2$, $a_{ht}(\Lambda_n)=1$.

\begin{proposition}\label{prop:hereditary-tilted-algebras}
We have
\begin{align*}
\label{equ:hereditary-tilted}
a_{ht}(\Lambda_n)&=2^{n-2}+(1-(-1)^{n})\times 2^{[\frac{n}{2}]-2}
=\begin{cases}
2^{n-2}, & \text{if $n$ is even},\\ 
2^{n-2}+2^{\frac{n-3}{2}}, & \text{if $n$ is odd}.
\end{cases}
\end{align*}
\end{proposition}
\begin{proof}
First the equality holds for $n=1$. Next, assume $n\geq 2$. 
By Lemma~\ref{lem:genealogical-tree} (a) and Lemma~\ref{lem:hereditary-QR}, we have when $n$ is even
\[
a_{ht}(\Lambda_n)=\frac{2^{n-1}}{2}=2^{n-2},
\]
and when $n$ is odd
\[
a_{ht}(\Lambda_n)=\frac{2^{n-1}-2^{\frac{n-1}{2}}}{2}+2^{\frac{n-1}{2}}=2^{n-2}+2^{\frac{n-3}{2}}.\qedhere
\]
\end{proof}


Now we can give the formula of $a_{t}(\Lambda_n)$, the number of isoclasses of basic tilted algebras of type $\Lambda_n$. Let $a_{nht}(\Lambda_n)$ be the number of isoclasses of basic non-hereditary tilted algebras of type $\Lambda_n$.

\begin{proposition}\label{prop:the-number-of-non-hereditary-tilted-algebras}
We have
\begin{align*}
a_{nht}(\Lambda_n)&=\frac{1}{n+1}(\begin{smallmatrix}2n\\n\end{smallmatrix})-2^{n-1},\\
a_{t}(\Lambda_n)&=\frac{1}{n+1}(\begin{smallmatrix}2n\\n\end{smallmatrix})+[1-(-1)^{n}]\times 2^{[\frac{n}{2}]-2}-2^{n-2}.
\end{align*}
\end{proposition}

\begin{proof} 
By Proposition~\ref{prop:number-of-tilting-modules-over-An} and Lemma~\ref{lem:genealogical-tree} (a), we have $|\cq(n)|=t(\Lambda_n)=\frac{1}{n+1}(\begin{smallmatrix}2n\\n\end{smallmatrix})$. By Lemma~\ref{lem:genealogical-tree} (b) and Lemma~\ref{lem:hereditary-QR} (a), we have
\[
a_{nht}(\Lambda_n)=|\cq_{nh}(n)|=|\cq(n)|-|\cq_h(n)|=\frac{1}{n+1}(\begin{smallmatrix}2n\\n\end{smallmatrix})-2^{n-1}.
\]
This in conjunction with Proposition~\ref{prop:hereditary-tilted-algebras} implies that
\[
a_{t}(\Lambda_n)=a_{ht}(\Lambda_n)+a_{nht}(\Lambda_n)=\frac{1}{n+1}(\begin{smallmatrix}2n\\n\end{smallmatrix})+[1-(-1)^{n}]\times 2^{[\frac{n}{2}]-2}-2^{n-2}.\qedhere
\]
\end{proof}

We believe that Proposition~\ref{prop:hereditary-tilted-algebras} and Proposition~\ref{prop:the-number-of-non-hereditary-tilted-algebras} are well-known, but unfortunately we fail to find them in the literature.

\subsubsection{Some subsets of $\ca_t(\Lambda_n)$}

Let $\ct^1(\Lambda_n)$ denote the set of isoclasses of  basic tilting modules over $\Lambda_n$ which contains $P(n)$ as a direct summand, $\ca_t^1(\Lambda_n)=\{\End(T)\mid T\in\ct^1(\Lambda_n)\}$, and $a_t^1(\Lambda_n)=|\ca_t^1(\Lambda_n)|$. Then $a_t^1(\Lambda_1)=1$.

\begin{lemma}
\label{lem:tilting-modules-and-tilted-algebras-with-P(n)}
\begin{itemize}
\item[(a)] The bijection $\varepsilon_1$ in Lemma~\ref{lem:genealogical-tree} (a) restricts to a bijection $\ct^1(\Lambda_n)\to\cq^1(n)$, with $P(n)$ corresponding to the vertex $u$.
\item[(b)] $\varepsilon\colon\ct^1(\Lambda_n)\to\ca_t^1(\Lambda_n)$ is bijective.
\item[(c)] For $n\geq 2$, we have $a_t^1(\Lambda_n)=t(\Lambda_{n-1})$.
\end{itemize}
\end{lemma}

\begin{proof}
(a) For $T\in\ct(\Lambda_n)$, the vertices of $\varepsilon_1(T)$ which correspond to indecomposable projective direct summands of $T$ form the left border of $\varepsilon_1(T)$. Among them, the vertex corresponding to $P(n)$ is the only possible vertex which is a source and has exactly one neighbour, by Remark~\ref{rmk:tilting-modules-of-Lambdan}. 

(b) Observe that elements of $\cq^1(n)$ are all non-hereditary except one: $\overrightarrow{\mathbb{A}_{n}}$ with the unique sink as the root, which corresponds to $\Lambda_n=P(1)\oplus\ldots\oplus P(n)$. It follows from Corollary~\ref{cor:non-hereditary-QR} that elements of $\cq^1(n)$ are pairwise non-isomorphic as quivers with relation. So  $\varepsilon\colon \ct^1(\Lambda_n)\to\ca_t^1(\Lambda_n)$ is bijective.

(c) We have $a_t^1(\Lambda_n)=|\ct^1(\Lambda_n)|=t(\Lambda_{n-1})$, where the first equality is by (b) and the second equality is by Remark~\ref{rek:tilting-module-over-A-turn-to-B}.
\end{proof}

For $n\geq 3$, let $\ct^2(\Lambda_n)$ be the set of isoclasses of basic tilting modules over $\Lambda_n$ which contains $P(n)$ as a direct summand but does not contain $P(n-1)$ as a direct summand and $\ca_t^2(\Lambda_n)=\{\End(T)\mid T\in\ct^2(\Lambda_n)\}$.

\begin{lemma}
\label{lem:number-of-tilting-modules-without-P(n-1)-as-summand}
The bijection $\varepsilon_1$ in Lemma~\ref{lem:genealogical-tree} (a) restricts to a bijection $\ct^2(\Lambda_n)\to\cq^2(n)$. Moreover, $|\cq^2(n)|=|\ct^2(\Lambda_n)|=t(\Lambda_{n-1})-t(\Lambda_{n-2})$ and $|\cq^3(n)|=t(\Lambda_{n-1})-t(\Lambda_{n-2})-2^{n-2}+1$.
\end{lemma}
\begin{proof}
Let $T\in\ct^1(\Lambda_n)$. Let $2\leq j\leq n-1$ be the maximal such that $P(j)$ is a direct summand of $T$ and $v$ be the corresponding vertex in $\varepsilon_1(T)$. Then there is an arrow $u\to v$. If $j=n-1$, then $v$ does not lie on a relation, because a zero homomorphism from $P(n)$ which factors through $P(n-1)$ non-trivially is of the form $P(n)\to P(n-1)\to S(n-1)$, but $S(n-1)$ cannot be a direct summand of $T$. If $j\leq n-2$, then $v$ lies on a relation, because $\tau^{-2}P(j+2)$ is a direct summand of $T$ by Lemma~\ref{lem:maximal-P(i)-to-relation}, and any composition $P(n)\to P(j)\to \tau^{-2}P(j+2)$ is zero. Then the first statement follows from Lemma~\ref{lem:tilting-modules-and-tilted-algebras-with-P(n)} (a). The first part of the second statement follows from the first statement by applying Remark~\ref{rek:tilting-module-over-A-turn-to-B} twice.  Let $R\in\cq^2(n)-\cq^3(n)$ and $R'$ be the quiver with relation obtained from $R$ by removing the vertex $u$. Then $R'\in\cq_h(n-1)$. Notice that all elements of $\cq_h(n-1)$ can be $R'$ except $\overrightarrow{\mathbb{A}}_{n-1}$ with the unique sink as the root. Therefore $|\cq^2(n)|-|\cq^3(n)|=2^{n-2}-1$, and we obtain the desired formula for $|\cq^3(n)|$.
\end{proof}

Let $\ca_t^4(\Lambda_n)$ denote the set of isoclasses of endomorphism algebras of basic tilting modules over $\Lambda_n$ which does not contain $P(n)$ as a direct summand and let $a_t^4(\Lambda_n)$ denote its cardinality. Then $a_t^4(\Lambda_1)=0$.

\begin{lemma} \label{lem:tilted-algebra-without-Pn}
For $n\geq 2$, we have
$a_t^4(\Lambda_{n})=a_{t}(\Lambda_{n})-t(\Lambda_{n-1})+1$.
\end{lemma}

\begin{proof}  By Lemma~\ref{lem:genealogical-tree} (b), the intersection $\ca_t^4(\Lambda_n)\cap\ca_t^1(\Lambda_n)$ contains hereditary tilted algebras only. By the proof of Lemma~\ref{lem:tilting-modules-and-tilted-algebras-with-P(n)} (b), in $\ca_t^1(\Lambda_n)$ there is only one hereditary tilted algebra, which is $\Lambda_n=\End(P(1)\oplus\ldots \oplus P(n))=\End(I(1)\oplus\ldots\oplus I(n))$ and thus belongs also to $\ca_t^4(\Lambda_n)$. Therefore $\ca_t^4(\Lambda_n)\cap\ca_t^1(\Lambda_n)=\{\Lambda_n\}$. 
Since $\ca_t(\Lambda_n)=\ca_t^4(\Lambda_n)\cup\ca_t^1(\Lambda_n)$, it follows that $a_t^4(\Lambda_{n})=a_{t}(\Lambda_{n})-a_t^1(\Lambda_n)+1=a_{t}(\Lambda_{n})-t(\Lambda_{n-1})+1$.
\end{proof}

For $n\geq 2$, let $\ca_t^5(\Lambda_n)$ denote the set of isoclasses of endomorphism algebras of basic tilting modules over $\Lambda_n$ which contains $P(2)$ as a direct summand and let $a_t^5(\Lambda_n)$ denote its cardinality.

\begin{lemma}
\label{lem:tilted-algebra-with-P2}
$a_t^5(\Lambda_n)=a_t(\Lambda_{n-1})$ and $|\ca_t^4(\Lambda_n)\cap\ca_t^5(\Lambda_n)|=a_t^4(\Lambda_{n-1})$.
\end{lemma}
\begin{proof}
This follows from Proposition~\ref{prop:tilting-modules-of-Lambdan} because the additive closure of the wing of $P(2)$ is naturally equivalent to $\mod \Lambda_{n-1}$.
\end{proof}

For $n=1,2$, put $\ca_t^6(\Lambda_n)=\ca_t(\Lambda_n)$, which has one element only. For $n\geq 3$, let $\ca_t^6(\Lambda_n)=\ca_t^4(\Lambda_n)\cup\{\End(P(1)\oplus\ldots\oplus P(n-2)\oplus P(n)\oplus \tau^{-2}P(n))\}$.  As this extra element belongs to $\ca_t^1(\Lambda_n)$ and is non-hereditary, it does not belong to $\ca_t^4(\Lambda_n)$. Therefore letting $a_t^6(\Lambda_n)$ denote the cardinality of $\ca_t^6(\Lambda_n)$, we have as a consequence of Lemma~\ref{lem:tilted-algebra-without-Pn}

\begin{lemma} \label{prop:tilted-algebra-under-mutation}
$
a_t^{6}(\Lambda_{n})=
\begin{cases} 1 & \text{if } n=1,2,\\
a_{t}(\Lambda_{n})-t(\Lambda_{n-1})+2 & \text{if }n\geq 3.
\end{cases}
$
\end{lemma}

When $n=3$, the difference between $\ca_t(\Lambda_3)$ and $\ca_t^4(\Lambda_3)$ is the unique non-hereditary element in $\ca_t(\Lambda_3)$, and is exactly $\End(P(1)\oplus P(3)\oplus \tau^{-2}P(3))$. This means $\ca_t^6(\Lambda_3)=\ca_t(\Lambda_3)$.

\subsection{2-term silting complexes of $\Lambda_n$ and silted algebras}
\label{ss:silted-algebras-type-Lambdan}
In this subsection, we give a classification of the 2-term silting complexes over $\Lambda_n$ and prove Theorem~\ref{thm:the-number-of-silted-algebras-of-Lambda_n} as a consequence.

First, denote by $s_{2}(\Lambda_n)$ the number of isoclasses of basic $2$-term silting complexes over $\Lambda_n$. Then by \cite[Theorem 1]{ObaidNaumanFakiehRingel14} and Theorem~\ref{thm:correspondence-between-silting-and-support-tau-tilting} we have $s_{2}(\Lambda_n)=\frac{1}{n+2}(\begin{smallmatrix}2n+2\\n+1\end{smallmatrix})$. 
The following table contains the first values of $t(\Lambda_n)$ and $s_2(\Lambda_n)$.

$$\begin{tabular}{|c||c|c|c|c|c|c|c|c|c|c| p{<20>}}
\hline
$n$ & $1$ & $2$ & $3$ & $4$ & $5$ & $6$ & $7$ & $8$ & $9$ & $10$\\
\hline
$t(\Lambda_n)$ & 1 & 2 & 5 & 14 & 42 & 132 & 429 & 1430 & 4862 & 16796 \\
\hline
$s_2(\Lambda_n)$ & 2 & 5 & 14 & 42 & 132 & 429 & 1430 & 4862 & 16796 & 58786\\
\hline
\end{tabular} $$

\medskip
Next, we provide a classification. We need the following lemma. The notions \emph{upper ray}, \emph{lower ray} and \emph{hammock} directly generalise to the AR quiver of $K^b(\proj \Lambda_n)$.

\begin{lemma}
\label{lem:derived-Hom-and-hammock}
Let $X$ and $Y$ be indecomposable objects in $K^b(\proj \Lambda_n)$. Then $\Hom(X,Y)\neq 0$ if and only if $Y$ belongs to the hammock starting at $X$ if and only if $X$ belongs to the hammock starting at $\tau Y[-1]$. If these conditions are satisfied, then $\Hom(X,Y)=k$.
\end{lemma}

The AR-quiver of $K^b(\proj \Lambda_n)$ is $\mathbb{Z}\overrightarrow{\mathbb{A}}_{n}$, by \cite[Section I.5.6]{Happel88}. We consider the AR-quiver of $\mod \Lambda_n$ as a full subquiver, and then draw the AR-quiver of $K^{[-1,0]}(\proj \Lambda_n)$ in Firgure~2.
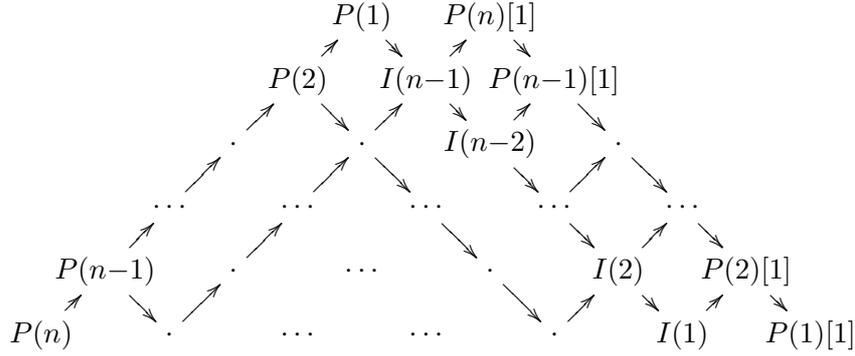
\begin{figure}
\label{fig:the-AR-quiver-of-2-term-complexes-over-Lambdan}
\hspace{10pt}\begin{xy}
0;<3pt,0pt>:<0pt,3pt>::
(56,-40)*+{P(1)[1]}="27",
(48,-32)*+{P(2)[1]}="26",
(40,-24)*+{\cdots}="25",
(32,-16)*+{\cdot}="24",
(24,-8)*+{P(n{-}1)[1]}="23",
(16,0)*+{P(n)[1]}="22",
 (0,0)*+{P(1)}="1",
(-8,-8)*+{P(2)}="2",
(8,-8)*+{I(n{-}1)}="3",
(-16,-16)*+{\cdot}="4",
(0,-16)*+{\cdot}="5",
(16,-16)*+{I(n{-}2)}="6",
(-24,-24)*+{\cdots}="7",
(-8,-24)*+{\cdots}="8",
(8,-24)*+{\cdots}="9",
(24,-24)*+{\cdots}="10",
(-32,-32)*+{P(n{-}1)}="11",
(-16,-32)*+{\cdot}="12",
(0,-32)*+{\cdots}="13",
(16,-32)*+{\cdot}="14",
(32,-32)*+{I(2)}="15",
(-40,-40)*+{P(n)}="16",
(-24,-40)*+{\cdot}="17",
(-8,-40)*+{\cdots}="18",
(8,-40)*+{\cdots}="19",
(24,-40)*+{\cdot}="20",
(40,-40)*+{I(1)}="21",
\ar"2";"1", \ar"2";"5", \ar"5";"3", \ar"1";"3", \ar"4";"2", \ar"3";"6",
\ar"7";"4", \ar"6";"10", \ar"5";"9", \ar"9";"14", \ar"8";"5",\ar"14";"20",\ar"17";"12",
\ar"16";"11", \ar"15";"21",
\ar"11";"7", \ar"11";"17", \ar"10";"15",\ar"20";"15", \ar"12";"8",\ar"3";"22",
\ar"6";"23",\ar"22";"23", \ar"23";"24", \ar"24";"25", \ar"25";"26", \ar"26";"27",
\ar"10";"24",\ar"15";"25",\ar"21";"26",
\end{xy}
\caption{The~Auslander--Reiten~quiver~of~$K^{[-1,0]}(\proj \Lambda_n)$}
\end{figure}

\begin{proposition}\label{prop:2-term-silting-complexes-of-Lambda_n}
A complex $M\in K^b(\proj \Lambda_n)$ is a basic 2-term silting complex if and only if it belongs to one of the following three families:
\begin{itemize}
\item[(1)] $T$, a basic tilting module over $\Lambda_n$;

\item[(2)] $\tau^{-1} T$, where $T$ is a basic tilting module over $\Lambda_n$;

\item[(3)] $M=M'\oplus M''$, where $M'$ is a basic tilting module of the wing of $P(i)$ with $2\leq i\leq n$ and $M''$ is a basic tilting module of the wing of $P(i-1)[1]$.
\end{itemize}
\end{proposition}

\begin{proof} The `if' part is straightforward. We prove the `only if' part. 
First notice that if $M$ is a presilting (respectively, silting) complex belonging to the wing of $P(1)$, then $M$ is a pretilting (respectively, tilting) module. Next, let $M$ be a basic 2-term silting complex. Suppose that $1\leq i\leq n$ is the minimal integer such that $P(i)$ is a direct summand of $M$.  Write $M=M'\oplus M''$, where $M'$ is the direct sum of the indecomposable direct summands which belongs to the wing of $P(i)$. Then $\Hom(M'',P(i)[1])=0$. This implies, by Lemma~\ref{lem:derived-Hom-and-hammock}, that $M''$ belongs to the wing of $P(i-1)[1]$. Therefore $M'$ is presilting in the wing of $P(i)$ and $M''$ is presilting in the wing of $P(i-1)[1])$, so they must be tilting in the corresponding wings. Moreover, $|M'|\leq n-i+1$, $|M''|\leq i-1$ and $|M'|+|M''|=n$. It follows that  $|M'|=n-i+1$, $|M''|=i-1$ and that $M'$ and $M''$ are tilting in the corresponding wings. Now if $i=1$, then $M''=0$ and we are in case (1); otherwise, we are in case (3). If $i$ does not exist, then $\tau M$ has to be a basic tilting module and  we are in case (2).
\end{proof}


\begin{proof}[Proof of Theorem~\ref{thm:the-number-of-silted-algebras-of-Lambda_n}]
It follows from Proposition~\ref{prop:2-term-silting-complexes-of-Lambda_n} that $\ca_s(\Lambda_n)$ contains $\ca_t(\Lambda_n)$ as a subset. Moreover, if $M=M'\oplus M''$ belongs to family (3) in Proposition~\ref{prop:2-term-silting-complexes-of-Lambda_n}, then $\Hom(M',M'')=0=\Hom(M'',M')$ and $\End(M)=\End(M')\times\End(M'')$. Therefore
\begin{align*}
\ca_s(\Lambda_n)=\ca_t(\Lambda_n)\sqcup \ca_{ns}(\Lambda_n),
\end{align*}
and
\begin{align*}
\ca_{ns}(\Lambda_n)&=\bigcup_{i=2}^n \ca_t(\Lambda_{n-i+1})\times_s\ca_t(\Lambda_{i-1})=\bigsqcup_{m=1}^{[\frac{n}{2}]} \ca_t(\Lambda_m)\times_s \ca_t(\Lambda_{n-m})\\
&=\begin{cases} \bigsqcup_{m=1}^k \ca_t(\Lambda_m)\times \ca_t(\Lambda_{n-m}), & \text{if}~n=2k+1,\\
(\bigsqcup_{m=1}^{k-1} \ca_t(\Lambda_m)\times \ca_t(\Lambda_{n-m}))\sqcup (\ca_t(\Lambda_k)\times_s\ca_t(\Lambda_k)), &\text{if}~n=2k.
\end{cases}
\end{align*}
Thus (a) holds true and (b) follows from the equalities above and Proposition~\ref{prop:the-number-of-non-hereditary-tilted-algebras} and Lemma~\ref{lem:symmetric-product}.
\end{proof}

\section{Silted algebras of type $\Gamma_{n}$}
\label{s:silted-algebra-mutated-orientation}

In this section, we classify up to isomorphism the basic silted algebras of type $\Gamma_{n}$ ($n\geq 2$), the path algebra of the quiver
\begin{center}
$\begin{xy}
(-20,0)*+{Q\colon}="0",
(-12,0)*+{1}="1",
(0,0)*+{2}="2",
(12,0)*+{\cdots}="3",
(28,0)*+{n-1}="4",
(42,0)*+{n.}="5",
\ar"2";"1", \ar"3";"2",\ar"4";"3",\ar"4";"5",
\end{xy}$
\end{center}
Then we give a recurrence formula to compute the number of these silted algebras. 

\smallskip
Throughout this section, assume $n\geq 2$. 

\subsection{Main result} 
Put
\begin{align*}
\ca_s(\Gamma_n)&=\{\text{basic silted algebras of type $\Gamma_n$}\}/\cong,
\end{align*}
and let $a_s(\Gamma_n)$ denote its cardinality.

\begin{theorem} 
\label{thm:silted-algebras-of-Gamman}
$\ca_s(\Gamma_n)=\cb_1\sqcup\cb_2\sqcup\cb_3\sqcup\cb_4\sqcup\cb_5$, where
\begin{itemize}
  \item [(1)] $\cb_1=\ca_t(\Lambda_n)$;
  \item [(2)] for $n\leq 4$, $\cb_2=\emptyset$, and for $n\geq 5$, $\cb_2$ consists of isoclasses of path algebras of elements in $\cm(n)$ (defined in Section~\ref{ss:an-operation});
  \item [(3)] $\cb_3=\bigcup_{m=1}^{n-1}(\ca_t^6(\Lambda_m)\times_s \ca_t(\Lambda_{n-m}))$;
  \item [(4)] $\cb_4=\bigsqcup_{m=5}^{n-1}(\cb_2(\Gamma_m)\times \ca_t(\Lambda_{n-m}))$;
  \item [(5)] $\cb_5=\bigcup_{m=1}^{n-2}(\ca_t^1(\Lambda_{m})\times_s \ca_t(\Lambda_{n-m-1}))\times\ca_t(\Lambda_1)$.
\end{itemize}
\end{theorem}

In particular, all elements in $\ca_s(\Gamma_n)$ are tilted algebras, or equivalently, there are no strictly shod algebras in $\ca_s(\Gamma_n)$. The equality $\ca_s(\Gamma_n)=\cb_1\sqcup\cb_2\sqcup\cb_3\sqcup\cb_4\sqcup\cb_5$ is a partition of $\ca_s(\Gamma_n)$ according the tilting type of its elements: elements in $\cb_1$ are tilted algebras of type $\Lambda_n$, elements in $\cb_2$ are tilted algebras of type $\mathbb{A}_n$ but not of type $\Lambda_n$, elements in $\cb_3$ are tilted algebras of type $\Lambda_{m}\times \Lambda_{n-m}$ ($1\leq m\leq n-1$), elements of $\cb_4$ are tilted of type $\mathbb{A}_m\times \Lambda_{n-m}$ but not of type $\Lambda_m\times\Lambda_{n-m}$ ($5\leq m\leq n-1$) and elements of $\cb_5$ are tilted algebras of type $\Lambda_m\times\Lambda_{n-m-1}\times\Lambda_1$.
The proof of Theorem~\ref{thm:silted-algebras-of-Gamman} will be given in Section~\ref{sss:proof-of-main-theorem}.

\medskip

Next we give formulas for the cardinalities $b_1,\ldots,b_5$ of $\cb_1,\ldots,\cb_5$.

\begin{theorem}
\label{thm:number-of-silted-algebras-of-Gamman}
We have $a_s(\Gamma_n)=b_1+b_2+b_3+b_4+b_5$, and
\begin{itemize} 
\item[(1)] $b_1=a_t(\Lambda_n)$, which is given in Proposition~\ref{prop:the-number-of-non-hereditary-tilted-algebras};
\item[(2)] $b_2=\begin{cases} 0 & \text{ if } n\leq 4,\\
2 & \text{ if } n=5,\\
t(\Lambda_{n-1})-t(\Lambda_{n-2})-2^{n-2}-2^{n-4}+\frac{n-2}{2} & \text{ if $n\geq 6$ is even},\\
t(\Lambda_{n-1})-t(\Lambda_{n-2})-2^{n-2}-2^{n-4}+2^{\frac{n-3}{2}-1}+\frac{n-3}{2} & \text{ if $n\geq 7$ is odd}.
\end{cases}$
\item[(3)] $b_3=\sum_{m=1}^{[\frac{n}{2}]}b_{3,m}$, where
\[
b_{3,m}=
\begin{cases}
a_t(\Lambda_m)a_t(\Lambda_{n-m}) & \text{if } m<\mathrm{min}\{4,\frac{n}{2}\},\\
a_t^6(\Lambda_m)a_t(\Lambda_{n-m})+a_t^6(\Lambda_{n-m})a_t(\Lambda_m)\\
\hspace{30pt}-a_t^6(\Lambda_m)a_t^6(\Lambda_{n-m}) & \text{if }4\leq m<\frac{n}{2}, \\
a_t^6(\Lambda_{\frac{n}{2}})a_t(\Lambda_{\frac{n}{2}})-\frac{a_t^6(\Lambda_{\frac{n}{2}})(a_t^6(\Lambda_{\frac{n}{2}})-1)}{2} &\text{if } m=\frac{n}{2}~(\text{and $n$ is even});
\end{cases}
\]
\item[(4)] $b_4=\sum_{m=5}^{n-1}b_2(m)\times a_t(\Lambda_{n-m})$, where $b_2(m)=|\cb_2(\Gamma_m)|$;
\item[(5)] $b_5=\sum_{m=1}^{[\frac{n-1}{2}]} b_{5,m}$, where
\[
b_{5,m}=\begin{cases}
a_t(\Lambda_{n-2}) & \text{ if } m=1,\\
t(\Lambda_{m-1})a_t(\Lambda_{n-m-1})+t(\Lambda_{n-m-2})a_t(\Lambda_m)&\\
\hspace{30pt}-t(\Lambda_{m-1})t(\Lambda_{n-m-2}) & \text{ if } 1<m< \frac{n-1}{2},\\
t(\Lambda_{\frac{n-3}{2}})a_{t}(\Lambda_{\frac{n-1}{2}})-\frac{t(\Lambda_{\frac{n-3}{2}})\times(t(\Lambda_{\frac{n-3}{2}})-1)}{2} & \text{ if } m=\frac{n-1}{2} \text{and $n\geq 5$ is odd}.
\end{cases}
\]
\end{itemize}
\end{theorem}

The proof of Theorem~\ref{thm:number-of-silted-algebras-of-Gamman} will be given in Section~\ref{ss:cardinatities}. 
The following table contains the first values of $b_1,\ldots,b_5$ and $a_s(\Gamma_n)$.

$$\begin{tabular}{|c||c|c|c|c|c|c|c|c|c| p{<20>}}
\hline
$n$ & $2$ & $3$ & $4$ & $5$ & $6$ & $7$ & $8$ & $9$ & $10$\\
\hline
$b_1$ & 1 & 4 & 10 & 36 & 116 & 401 & 1366 & 4742 & 16540 \\
\hline
$b_2$ & 0  & 0 & 0 & 2  & 10  & 54 & 220 & 848 & 3116 \\
\hline
$b_3$ & 1  & 1 & 5 & 14  & 56  & 192 & 710 & 2555 & 9340\\
\hline
$b_4$ & 0  & 0 & 0 & 0  & 2  & 12 & 72 & 334 & 1456\\
\hline
$b_5$ & 0  & 1 & 1 & 5  & 14  & 53 & 182 & 657 & 2333\\
\hline
$a_s(\Gamma_n)$ & 2  & 6 & 16 & 57  & 198  & 712 & 2550 & 9136 & 32785 \\
\hline
\end{tabular} $$

\subsection{Strategy of the proof}
\label{ss:strategy}
Our proof of of Theorems~\ref{thm:silted-algebras-of-Gamman} and~\ref{thm:number-of-silted-algebras-of-Gamman} depend heavily on the classification of two-term silting complexes over $\Lambda_n$ and the classification of tilted algebras of type $\Lambda_n$.

Note that $Q$ is obtained from $\overrightarrow{\mathbb{A}}_{n}$ by reversing the arrow starting from the unique source $n$. So there is a BGP reflection functor $F\colon K^{b}(\proj \Gamma_{n})\xrightarrow{\simeq}K^{b}(\proj \Lambda_n)$, see \cite{BernsteinGelfandPonomarev73}. It follows that the AR-quiver of $K^{[-1,0]}(\proj \Gamma_{n})$ can be identified with the full subquiver drawn in Figure~\ref{fig:AR-quiver-Gamman} of the AR-quiver of $K^b(\proj \Lambda_n)$ and $K^{[-1,0]}(\proj \Gamma_{n})$ can be identified with the additive closure in $K^{b}(\proj \Lambda_n)$ of the indecomposable objects in this AR-quiver. We will use this AR-quiver to study the silted algebras of type $\Gamma_n$. For example, when we say $M$ is a 2-term silting complex over $\Gamma_n$, we mean that $M$ is a silting complex over $\Lambda_n$ whose direct summands belong to this subquiver. We also identify the AR-quiver of $\mod \Gamma_n$ with its full subquiver in the framed area. Note that there is a distinguished vertex $X=\tau^{-1}P(n)[1]=S(n-1)[1]$. For an indecomposable object $Y$, it follows from Lemma~\ref{lem:derived-Hom-and-hammock} that $\Hom(Y,X)\neq 0$ if and only if $Y$ belongs to the upper ray of $S(n-2)$, and in this case $\Hom(Y,X)=k$.

\begin{figure}[htbp]
\begin{tikzpicture}[scale=1]
\draw (0,0) node{$P(1)$};
\draw (2,0) node{$P(n)[1]$};
\draw (4,0) node{$X$};
\draw (-1,-1) node{$P(2)$};
\draw (1,-1) node{$I(n-1)$};
\draw (3,-1) node{$P(n-1)[1]$};
\draw (-2.15,-2.15) node{\begin{rotate}{45}$\cdots$\end{rotate}};
\draw (0,-2) node{$\tau^{-1}P(3)$};
\draw (2,-2) node{$I(n-2)$};
\draw (4,-2) node{$P(n-2)[1]$};
\draw (-3,-3) node{$P(n-2)$};
\draw (-1.15,-3.15) node{\begin{rotate}{45}$\cdots$\end{rotate}};
\draw (1,-3) node{$\tau^{-1}P(4)$};
\draw (3,-3) node{$I(n-3)$};
\draw (5.05,-2.85) node{$\ddots$};
\draw (-4,-4) node{$P(n-1)$};
\draw (-2,-4) node{$\tau^{-1}P(n-1)$};
\draw (-0.15,-4.15) node{\begin{rotate}{45}$\cdots$\end{rotate}};
\draw (2,-4) node{$\cdots$};
\draw (4.05,-3.85) node{$\ddots$};
\draw (6,-4) node{$P(2)[1]$};
\draw (-3,-5) node{$S(n-1)$};
\draw (-1,-5) node{$S(n-2)$};
\draw (1,-5) node{$\cdots$};
\draw (3,-5) node{$\cdots$};
\draw (5,-5) node{$I(1)$};
\draw (7,-5) node{P(1)[1]};
\draw[line width=0.6pt][->] (-0.8,-0.7) -- (-0.3,-0.2);
\draw[line width=0.6pt][->] (1.2,-0.7) -- (1.7,-0.2);
\draw[line width=0.6pt][->] (3.2,-0.7) -- (3.7,-0.2);
\draw[line width=0.6pt][->] (0.3,-0.2) -- (0.8,-0.7);
\draw[line width=0.6pt][->] (2.3,-0.2) -- (2.8,-0.7);

\draw[line width=0.6pt][->] (-0.7,-1.2) -- (-0.2,-1.7);
\draw[line width=0.6pt][->] (1.3,-1.2) -- (1.8,-1.7);
\draw[line width=0.6pt][->] (3.3,-1.2) -- (3.8,-1.7);

\draw[line width=0.6pt][->] (-1.7,-2.2) -- (-1.2,-2.7);
\draw[line width=0.6pt][->] (0.3,-2.2) -- (0.8,-2.7);
\draw[line width=0.6pt][->] (2.3,-2.2) -- (2.8,-2.7);
\draw[line width=0.6pt][->] (4.3,-2.2) -- (4.8,-2.7);

\draw[line width=0.6pt][->] (-2.7,-3.2) -- (-2.2,-3.7);
\draw[line width=0.6pt][->] (-0.7,-3.2) -- (-0.2,-3.7);
\draw[line width=0.6pt][->] (1.3,-3.2) -- (1.8,-3.7);
\draw[line width=0.6pt][->] (3.3,-3.2) -- (3.8,-3.7);
\draw[line width=0.6pt][->] (5.3,-3.2) -- (5.8,-3.7);

\draw[line width=0.6pt][->] (-3.7,-4.2) -- (-3.2,-4.7);
\draw[line width=0.6pt][->] (-1.7,-4.2) -- (-1.2,-4.7);
\draw[line width=0.6pt][->] (0.3,-4.2) -- (0.8,-4.7);
\draw[line width=0.6pt][->] (2.3,-4.2) -- (2.8,-4.7);
\draw[line width=0.6pt][->] (4.3,-4.2) -- (4.8,-4.7);
\draw[line width=0.6pt][->] (6.3,-4.2) -- (6.8,-4.7);

\draw[line width=0.6pt][->] (-3.7,-4.2) -- (-3.2,-4.7);
\draw[line width=0.6pt][->] (-1.7,-4.2) -- (-1.2,-4.7);
\draw[line width=0.6pt][->] (0.3,-4.2) -- (0.8,-4.7);
\draw[line width=0.6pt][->] (2.3,-4.2) -- (2.8,-4.7);
\draw[line width=0.6pt][->] (4.3,-4.2) -- (4.8,-4.7);
\draw[line width=0.6pt][->] (6.3,-4.2) -- (6.8,-4.7);

\draw[line width=0.6pt][->] (-1.8,-1.7) -- (-1.3,-1.2);
\draw[line width=0.6pt][->] (0.2,-1.7) -- (0.7,-1.2);
\draw[line width=0.6pt][->] (2.2,-1.7) -- (2.7,-1.2);

\draw[line width=0.6pt][->] (-2.8,-2.7) -- (-2.3,-2.2);
\draw[line width=0.6pt][->] (-0.8,-2.7) -- (-0.3,-2.2);
\draw[line width=0.6pt][->] (1.2,-2.7) -- (1.7,-2.2);
\draw[line width=0.6pt][->] (3.2,-2.7) -- (3.7,-2.2);

\draw[line width=0.6pt][->] (-3.8,-3.7) -- (-3.3,-3.2);
\draw[line width=0.6pt][->] (-1.8,-3.7) -- (-1.3,-3.2);
\draw[line width=0.6pt][->] (0.2,-3.7) -- (0.7,-3.2);
\draw[line width=0.6pt][->] (2.2,-3.7) -- (2.7,-3.2);
\draw[line width=0.6pt][->] (4.2,-3.7) -- (4.7,-3.2);

\draw[line width=0.6pt][.] (3.2,0.5) -- (-0.5,0.5);
\draw[line width=0.6pt][.] (3.2,0.5) -- (1.8,-1);
\draw[line width=0.6pt][.] (1.8,-1) -- (6.3,-5.5);
\draw[line width=0.6pt][.] (6.3,-5.5) -- (-3.5,-5.5);
\draw[line width=0.6pt][.] (-3.5,-5.5) -- (-5,-4);
\draw[line width=0.6pt][.] (-0.5,0.5) -- (-5,-4);
\end{tikzpicture}
\caption{The Auslander--Reiten quiver of $K^{[-1,0]}(\proj \Gamma_n)$}
\label{fig:AR-quiver-Gamman}
\end{figure}
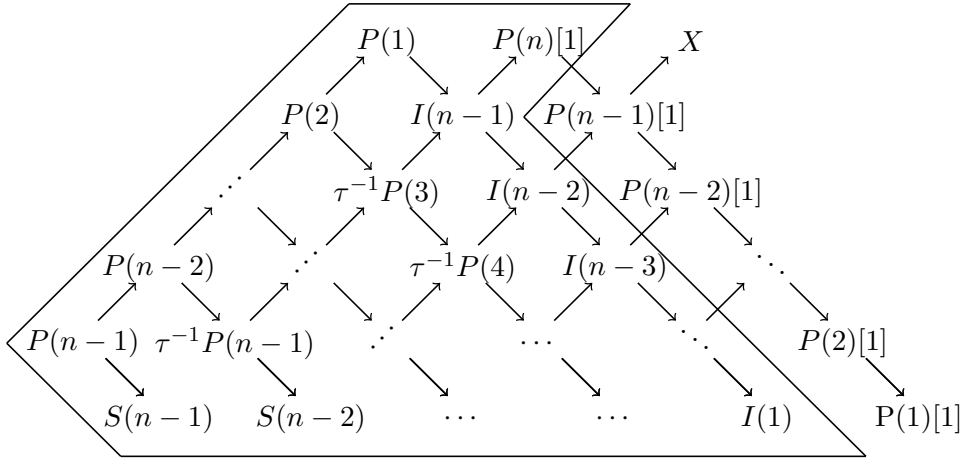

Let $T\in K^{[-1,0]}(\proj \Lambda_n)$ be a basic presilting complex which does not contain $P(n)$ as a direct summand. Put 
$M=T\oplus X$ and $N=T\oplus P(n)$. We will write $M=N^*$ and $N=M^*$ and view $(-)^*$ as an operator.

\begin{proposition}\label{prop:equ-of-2-term-silting-complexes}
$M$ is a basic 2-term silting complex over $\Gamma_{n}$ if and only if $M^*$ is a basic 2-term silting complex over $\Lambda_n$.
\end{proposition}

\begin{proof}
We need to show that
\begin{center}
$\Hom(T\oplus P(n),(T\oplus P(n))[1])=0\Longleftrightarrow\Hom(T\oplus X,(T\oplus X)[1])=0$.
\end{center}
Since $\Hom(T,T[1])=0$, $\Hom(X,X[1])=0$ and $\Hom(P(n),P(n)[1])=0$, we only need to prove that
\begin{center}
$\Hom(T,P(n)[1])\oplus \Hom(P(n),T[1])=0\Longleftrightarrow\Hom(T,X[1])\oplus \Hom(X,T[1])=0$.
\end{center}
This follows from the facts $\Hom(P(n),T[1])=0$, $\Hom(T,X[1])\cong D\Hom(X,\tau T)=0$, and $\Hom(X,T[1])\cong D\Hom(T,\tau X)=D\Hom(T,P(n)[1])$.
\end{proof}

Using $X$ and $(-)^*$ we divide basic 2-term silting complexes $M$ over $\Gamma_n$ into the following three families:
\begin{itemize}
\item[(I)] $M$ is a basic 2-term silting complex over $\Gamma_n$ such that $X$ is not a direct summand of $M$;
\item[(II)] $M$ is a basic 2-term silting complex over $\Gamma_n$ such that $X$ is a direct summand of $M$ and $M^*$ is a tilting module over $\Lambda_n$;
\item[(III)] $M$ is a basic 2-term silting complex over $\Gamma_n$ such that $X$ is a direct summand of $M$ and $M^*$ is not a tilting module over $\Lambda_n$.
\end{itemize}
For $k=\mathrm{I,II,III}$, put
\begin{align*}
\ca_s^k(\Gamma_n)&=\{\End(M)\mid \text{$M$ belongs to the family (k)}\}/\cong.
\end{align*}
It is clear that 
\[
\ca_s(\Gamma_n)=\ca_s^{\rm I}(\Gamma_n)\cup \ca_s^{\rm II}(\Gamma_n)\cup\ca_s^{\rm III}(\Gamma_n).
\]
In Section~\ref{sec:silted-algebras-of-Gamman}, we will classify basic 2-term silting complexes in the families (I), (II) and (III), respectively, and describe $\ca_s^{\rm I}(\Gamma_n)$, $\ca_s^{\rm II}(\Gamma_n)$ and $\ca_s^{\rm III}(\Gamma_n)$. We see as a consequence that all these silted algebras are tilted algebras. Then we partition $\ca_s(\Gamma_n)$ according to the tilting type of its elements to obtain $\cb_1,\ldots,\cb_5$, thus proving Theorem~\ref{thm:silted-algebras-of-Gamman}. In Section~\ref{ss:cardinatities} we will calculate the cardinalities of $\cb_1,\ldots,\cb_5$ and prove Theorem~\ref{thm:number-of-silted-algebras-of-Gamman}. 

\subsection{The classification of silted algebras of type $\Gamma_{n}$}\label{sec:silted-algebras-of-Gamman}
In this subsection we classify basic 2-term silting complexes over $\Gamma_n$ according to the families (I), (II) and (III), and compute their endomorphism algebras. 

\subsubsection{Family {\rm (I)}}
\label{sss:family-I}
Let $M$ be a basic 2-term silting complex over $\Gamma_n$ and assume that  $X$ is not a direct summand of $M$. Then $M$ is a 2-term silting complex over $\Lambda_n$. So by Proposition~\ref{prop:2-term-silting-complexes-of-Lambda_n}, $M$ belongs to one of the following three families:
\begin{itemize}
\item[(I1)] $T$, a basic tilting module over $\Lambda_n$ which does not contain $P(n)$ as a direct summand;

\item[(I2)] $\tau^{-1} T$, where $T$ is a basic tilting module over $\Lambda_n$;

\item[(I3)] $M=M'\oplus M''$, where $M'$ is a basic tilting module of the wing of $P(i)$ with $2\leq i\leq n-1$ which does not contain $P(n)$ as a direct summand and $M''$ is a basic tilting module of the wing of $P(i-1)[1]$.
\end{itemize}
Notice that the family (I3) is empty if $n=2$. Let $\ca_s^{\rm{I},1}(\Gamma_n)$ (respectively, $\ca_s^{\rm{I},2}(\Gamma_n)$) be the set of isoclasses of $\End(M)$, where $M$ belongs to the first two families (respectively, the third family). Then
\begin{align}
\label{eqn:ca-I}
\ca_s^{\rm I}(\Gamma_n)&=\ca_s^{\rm{I},1}(\Gamma_n)\sqcup\ca_s^{\rm{I},2}(\Gamma_n),\nonumber\\
\ca_s^{\rm{I},1}(\Gamma_n)&=\ca_t^4(\Lambda_n)\cup\ca_t(\Lambda_n)=\ca_t(\Lambda_n),\\
\ca_s^{\rm{I},2}(\Gamma_n)&=\bigcup_{m=2}^{n-1}\ca_t^4(\Lambda_m)\times_s \ca_t(\Lambda_{n-m}).\nonumber
\end{align}
We recall that $\ca_t^4(\Lambda_n)$ is the set of isoclasses of endomorphism algebras of basic tilting modules over $\Lambda_m$ which does not contain $P(n)$ as a direct summand. 

As corollaries, we have

\begin{corollary}\label{cor:tilted-are-silted-of-B}
All tilted algebras of type $\Lambda_n$ are silted of type $\Gamma_{n}$.
\end{corollary}

\begin{corollary}
\label{cor:number-of-tilted-algebra-of-tyep-Gamman}
For $n\geq 3$, the number of tilted algebras of type $\Gamma_{n}$ is
$$a_{t}(\Gamma_{n})=a_{t}(\Lambda_n) - t(\Lambda_{n{-}1}) +t(\Lambda_{n{-}2}).$$
\end{corollary}
\begin{proof}
Assume that $M$ is a basic tilting module over $\Gamma_n$. Then $M$ is a basic 2-term silting complex over $\Gamma_n$ which does not have $X,~P(n-1)[1],\ldots,P(1)[1]$ as direct summand. So $M$ belongs to one of the following two families:
\begin{itemize}
\item[(1)] $T$, a basic tilting module over $\Lambda_n$ which does not contain $P(n)$ as a direct summand;

\item[(2)] $\tau^{-1} (P(1)\oplus T')$, where $T'$ is a basic tilting module in the wing of $P(2)$;
\end{itemize}
Thus $\ca_t(\Gamma_n)$ is a proper subset of $\ca_t(\Lambda_n)$, precisely, $\ca_t(\Gamma_n)=\ca_t^4(\Lambda_n)\cup\ca_t^5(\Lambda_n)$. So 
\begin{align*}
a_t(\Gamma_n)&=a_t^4(\Lambda_n)+a_t^5(\Lambda_n)-|\ca_t^4(\Lambda_n)\cap\ca_t^5(\Lambda_n)|\\
&=a_{t}(\Lambda_n) - t(\Lambda_{n{-}1}) +t(\Lambda_{n{-}2}),
\end{align*} 
as desired. Here the last equality follows by Lemma~\ref{lem:tilted-algebra-with-P2}.
\end{proof}

\subsubsection{Family \rm{(II)}}
\label{sss:family-II}

Let $M$ be a basic 2-term silting complex over $\Gamma_n$ and assume that $X$ is a direct summand of $M$ and that $M^*$ is a tilting module over $\Lambda_n$. Write $M=M'\oplus X$ and $M^*=M'\oplus P(n)$. Then $\Hom(X,M')=0$.


Write $M^*=P(1)\oplus T'\oplus T''$ according to Proposition \ref{prop:tilting-modules-of-Lambdan} (2), where $T'$ is a tilting module in the wing of $P(i)$ for some $2\leq i\leq n$ containing $P(n)$ as a direct summand and $T''$ is a tilting module in the wing of $I(i-2)$. There are two cases:

Case 1: $T'$ contains $P(n{-}1)$ as a direct summand. Denote by $\ca_s^{\rm{II},1}(\Gamma_n)$ the set of isoclasses of $\End(M)$ for such $M$'s. By Lemma~\ref{lem:derived-Hom-and-hammock} $M^*$ has no direct summands in the upper ray of $S(n-2)$, and hence $\Hom(M',X)=0$. So
$$\End(M)=\End(M')\times \End(X).$$ 
Moreover, by Remark \ref{rek:tilting-module-over-A-turn-to-B}, $\End(M')$ is the endomorphism algebra of a tilting module over $\Lambda_{n-1}$ which contains $P(n-1)$ as a direct summand, and any such tilted algebra of type $\Lambda_{n-1}$ can occur here. So
\[
\ca_s^{\rm{II},1}(\Gamma_n)=\ca_t^1(\Lambda_{n-1})\times_s\ca_t(\Lambda_1) =\ca_t^1(\Lambda_{n-1})\times \ca_t(\Lambda_1).
\]

Case 2: $T'$ does not contain $P(n{-}1)$ as a direct summand. This implies $n\geq 3$. Moreover, $\Hom(X,M'[-1])=0$, and by Lemma~\ref{lem:derived-Hom-and-hammock}, 
\begin{align*}
\Hom(M',X[-1])=\Hom(M',S(n-1))=0.
\end{align*}
So $M$ is a tilting complex over $\Lambda_n$, and hence a 2-term tilting complex over $\Gamma_n$. It then follows from Corollary~\ref{cor:endomorphism-algebra-of-2-term-tilting-is-tilted} that $\End(M)$ is tilted of type $\mathbb{A}_n$. They form two groups: we denote by $\ca_s^{\rm{II},2}(\Gamma_n)$ the set of those tilted of type $\Lambda_n$ and by $\ca_s^{\rm{II},3}(\Gamma_n)$ the set of those not tilted of type $\Lambda_n$.
 We divide this case further into three subcases:

\emph{Subcase (2.1):} $T'=P(n)$ and $T''$ is a tilting module in the wing of $I(n-2)$. Since $\Hom(I(n-2),X)=k$ and $\Hom(P(1),X)=0$, it follows that $\End(M)$ is the path algebra of the quiver of the form:
  $$
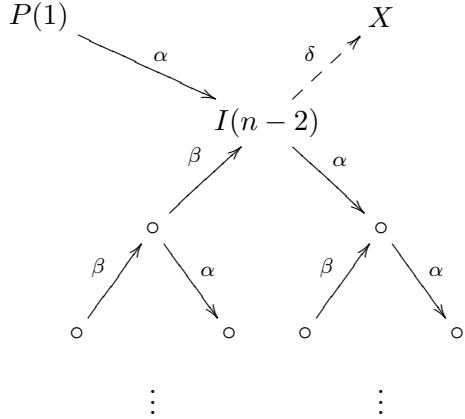
\begin{xy}
(0,0) *+{P(1)}="1",
(45,0) *+{X}="23",
(30,-14) *+{I(n-2)}="3",
(15,-28) *+{\circ}="6",
(45,-28) *+{\circ}="7",
(5,-42) *+{\circ}="12",
(25,-42) *+{\circ}="13",
(35,-42) *+{\circ}="14",
(55,-42) *+{\circ}="15",
(15,-50) *+{\vdots}="20",
(45,-50) *+{\vdots}="21",
\ar^{\alpha}"1";"3",
\ar^{\alpha}"3";"7",\ar^{\beta}"6";"3",\ar^{\beta}"12";"6",
\ar^{\alpha}"6";"13",\ar^{\beta}"14";"7",\ar^{\alpha}"7";"15",\ar@{-->}^{\delta}"3";"23",
\end{xy}$$
with all possible relations $\alpha\beta=0$ and $\delta\alpha=0$. 
This quiver with relation is clearly of the form $\mu(R)$ for some $R\in\cq^2(n)$ whose left border has exactly two vertices. By Lemma~\ref{lem:non-hereditary-mutated-QR-1}, it is of the form $E(R_1)$ for some $R_1\in\cq(n-2)$, and any $R_1\in \cq(n-2)$ can occur here.

\emph{Subcase (2.2):} $T'$ is a tilting module in the wing of $P(2)$ and $T''=0$. Let $2\leq j\leq n-2$ be maximal such that $M^*$ contains $P(j)$ as a direct summand. Then by Lemma \ref{lem:maximal-P(i)-to-relation}, we know that $M^*$ contains $\tau^{-2}P(j+2)$ but not $\tau^{-1}P(j+1)$ as a direct summand. Moreover, $M$ has no direct summands in the upper ray of $\tau^{-2}P(j+2)$ except $X$ and $\tau^{-2}P(j+2)$, by Lemma~\ref{lem:derived-Hom-and-hammock}. This shows that $\End(M)$ is the path algebra of the quiver of the form:
$$
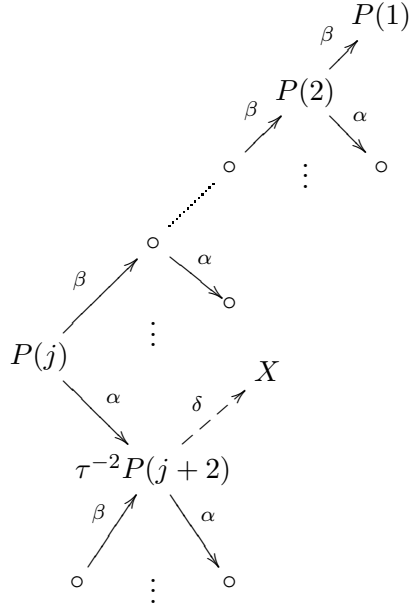
\begin{xy}
(20,30) *+{P(1)}="1",
(10,20) *+{P(2)}="3",
(0,10) *+{\circ}="7",
(20,10) *+{\circ}="8",
(10,10) *+{\vdots}="9",
(-10,0) *+{\circ}="4",
(0,-8) *{\circ}="6",
(-25,-15) *+{P(j)}="2",
(-10,-30) *+{\tau^{-2}P(j+2)}="5",
(-20,-45) *+{\circ}="10",
(0,-45) *+{\circ}="11",
(5,-17) *+{X}="16",
(-10,-45) *+{\vdots}="19",
(-10,-11) *+{\vdots}="20", 
\ar^{\beta}"3";"1",\ar^\beta"2";"4",\ar@{.}"4";"7",\ar^\alpha"4";"6",\ar^{\alpha}"2";"5",\ar^{\beta}"10";"5",\ar^{\alpha}"5";"11", \ar^\beta"7";"3", \ar^\alpha"3";"8",
\ar@{-->}^{\delta}"5";"16",
\end{xy}$$ 
with all possible relations $\alpha\beta=0$ and $\delta\alpha=0$. This quiver with relation is clearly of the form $\mu(R)$ for some $R\in\cq^2(n)$ whose root is a sink. By Lemma~\ref{lem:non-hereditary-mutated-QR-2}, it is of the form $E(R_1,R_2)$ for some $R_1\in\cq^1(j-1)$ and $R_2\in\cq(n-j-1)$, and any $R_1\in\cq^1(j-1)$ and $R_2\in\cq(n-j-1)$ can occur here.

\emph{Subcase (2.3):}  $T'$ is a tilting module of the wing in $P(i)$, where $3\leq i\leq n-2$, and $T''$ is a tilting module in the wing of $I(i-2)$. Let $i\leq j\leq n-2$ be maximal such that $M^*$ contains $P(j)$ as a direct summand. Similar to Subcase (2.2), $\End(M)$ is the path algebra of the quiver of the form
$$
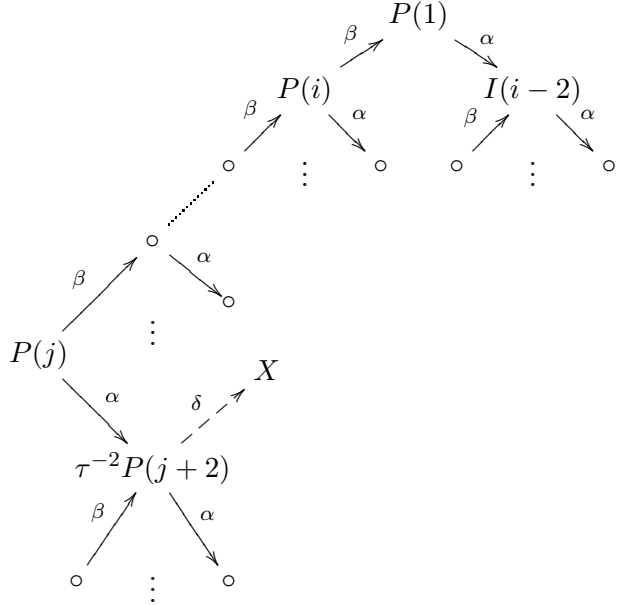
\begin{xy}
(25,30) *+{P(1)}="1",
(10,20) *+{P(i)}="3",
(0,10) *+{\circ}="7",
(20,10) *+{\circ}="8",
(10,10) *+{\vdots}="9",
(-10,0) *+{\circ}="4",
(0,-8) *{\circ}="6",
(-25,-15) *+{P(j)}="2",
(-10,-30) *+{\tau^{-2}P(j+2)}="5",
(-20,-45) *+{\circ}="10",
(0,-45) *+{\circ}="11",
(5,-17) *+{X}="16",
(-10,-45) *+{\vdots}="19",
(-10,-11) *+{\vdots}="20", 
(40,20) *+{I(i-2)}="21",
(30,10) *+{\circ}="22",
(50,10) *+{\circ}="23",
(40,10) *+{\vdots}="24",
\ar^{\beta}"3";"1",\ar^\beta"2";"4",\ar@{.}"4";"7",\ar^\alpha"4";"6",\ar^{\alpha}"2";"5",\ar^{\beta}"10";"5",\ar^{\alpha}"5";"11", \ar^\beta"7";"3", \ar^\alpha"3";"8",
\ar@{-->}^{\delta}"5";"16", 
\ar^\alpha"1";"21", \ar^(0.4)\beta"22";"21", \ar^\alpha"21";"23", 
\end{xy}$$ 
with all possible relations $\alpha\beta=0$ and $\delta\alpha=0$. 
This quiver with relation is clearly of the form $\mu(R)$ for some $R\in\cq^2(n)$ whose root is not a sink and whose left boarder has at least three vertices. By Lemma~\ref{lem:non-hereditary-mutated-QR-3}, it is of the form $E(R_1,R_2,R_3)$ for some $R_1\in\cq^1(j-i+1)$, $R_2\in\cq(n-j-1)$ and $R_3\in\cq(i-2)$, and any $R_1\in\cq^1(j-i+1)$, $R_2\in\cq(n-j-1)$ and $R_3\in\cq(i-2)$ can occur here.

\medskip
To summarise, we have
\begin{align}
\label{eqn:ca-II}
\ca_s^{\rm{II}}(\Gamma_n)=\ca_s^{\rm{II},1}(\Gamma_n)\sqcup \ca_s^{\rm{II},2}(\Gamma_n) \sqcup \ca_s^{\rm{II},3}(\Gamma_n).
\end{align}
Note that $\ca_s^{{\rm II},3}(\Gamma_n)$ is empy for $n\leq 4$, $\ca_s^{\rm{II},2}(\Gamma_2)$ is empty and $\ca_s^{\rm{II},2}(\Gamma_3)$ consists of one element, the path algebra of
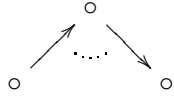
\[
\begin{xy}
 (-10,10)*+{\circ}="1",
(0,20)*+{\circ}="2",
(10,10)*+{\circ}="3",
(-4,16)*+{ }="4",
(4,16)*+{ }="5",
\ar"1";"2", \ar"2";"3", \ar@/^0.6pc/@{.}"5";"4",
\end{xy}
\]

\subsubsection{Family {\rm (III)}}
\label{sss:family-III}

Let $M$ be a basic 2-term silting complex over $\Gamma_n$ and assume that $X$ is a direct summand of $M$ and that  $M^*$ is a not tilting module over $\Lambda_n$.


According to Proposition \ref{prop:2-term-silting-complexes-of-Lambda_n} we can write $M^*=M_{1}^*\oplus M_{2}$, such that $M=M_1\oplus M_2$, $M_{1}^*$ is a tilting module in the wing of $P(i)$ with $2\leq i\leq n$ which has $P(n)$ as a direct summand and $M_{2}$ is a tilting module in the wing of $P(i-1)[1]$.
We know that $\Hom(M_{1}^*,M_{2})=0=\Hom(M_{2},M_{1}^*)$. Thus, there are the following two cases:

Case 1: $i=n$. Let $\ca_s^{\rm{III},1}(\Gamma_n)$ denote the set of isoclasses of $\End(M)$ for such $M$'s. In this case, $\tau^2(M)$ is a tilting module over $\Lambda_n$ which contains $P(2)$ as a direct summand. Therefore $\ca_s^{\rm{III},1}(\Gamma_n)=\ca_t^5(\Lambda_n)$.

Case 2: $i\neq n$. Let $\ca_s^{\rm{III},2}(\Gamma_n)$ denote the set of isoclasses of $\End(M)$ for such $M$'s. In this case, $\Hom(M_{2},X)=0$ and $\Hom(X,M_2)=0$, and hence $\Hom(M_{1},M_{2})=0=\Hom(M_{2},M_{1})$. Then
\[
\End(M)=\End(M_1)\times \End(M_{2}).
\]
Moreover, there are triangle equivalences $\thick(M_1)=\thick(M_1^*)\simeq K^b(\proj \Lambda_{n-i+1})\simeq K^b(\proj \Gamma_{n-i+1})$, $M_1$ can be considered as a 2-term silting complex over $\Gamma_{n-i+1}$ which has $X$ as direct summand.  Then $\ca_s^{\rm{III},2}(\Gamma_n)=\bigcup_{i=2}^{n-1}\ca_s^{\rm II}(\Gamma_{n-i+1})\times_s \ca_t(\Lambda_{i-1})$.

To summarise, we have
\begin{align}
\label{eqn:ca-III}
\ca_s^{\rm{III}}(\Gamma_n)&=\ca_s^{\rm{III},1}(\Gamma_n)\sqcup \ca_s^{\rm{III},2}(\Gamma_n),\nonumber\\
\ca_s^{\rm{III},1}(\Gamma_n)&=\ca_t^5(\Lambda_n),\\
\ca_s^{\rm{III},2}(\Gamma_n)&=\bigcup_{i=2}^{n-1}\ca_s^{\rm II}(\Gamma_{n-i+1})\times_s \ca_t(\Lambda_{i-1}))\nonumber\\
&=\bigcup_{i=2}^{n-1}(\ca_s^{\rm{II},1}(\Gamma_{n-i+1})\times_s\ca_t(\Lambda_{i-1}))
\sqcup \bigcup_{i=2}^{n-1} (\ca_s^{\rm{II},2}(\Gamma_{n-i+1})\times_s\ca_t(\Lambda_{i-1}))\nonumber\\
&\hspace{10pt}\sqcup \bigcup_{i=2}^{n-1} (\ca_s^{\rm{II},3}(\Gamma_{n-i+1})\times_s\ca_t(\Lambda_{i-1})).\nonumber
\end{align}

\subsubsection{The proof of Theorem \ref{thm:silted-algebras-of-Gamman}}
\label{sss:proof-of-main-theorem} 
We combine results in Subsections~\ref{sss:family-I}, \ref{sss:family-II} and \ref{sss:family-III} to prove Theorem~\ref{thm:silted-algebras-of-Gamman}.
Put 
\begin{align*}
\cb_1&=\ca_s^{\rm{I},1}(\Gamma_n)\cup \ca_s^{\rm{II},2}(\Gamma_n)\cup \ca_s^{\rm{III},1}(\Gamma_n),\\
\cb_2&=\ca_s^{\rm{II},3}(\Gamma_n),\\
\cb_3&=\ca_s^{\rm{I},2}(\Gamma_n)\cup\ca_s^{\rm{II},1}(\Gamma_n)\cup\bigcup_{i=2}^{n-1}(\ca_s^{\rm{II},2}(\Gamma_{n-i+1})\times_s \ca_t(\Lambda_{i-1})),\\
\cb_4&=\bigcup_{i=2}^{n-1}(\ca_s^{\rm{II},3}(\Gamma_{n-i+1})\times_s \ca_t(\Lambda_{i-1}))\\
\cb_5&=\bigcup_{i=2}^{n-1}(\ca_s^{\rm{II},1}(\Gamma_{n-i+1})\times_s\ca_t(\Lambda_{i-1}))
\end{align*}
Then $\ca_s(\Gamma_n)=\ca_s^{\rm I}(\Gamma_n)\cup \ca_s^{\rm II}(\Gamma_n)\cup\ca_s^{\rm III}(\Gamma_n)=\cb_1\sqcup\cb_2\sqcup\cb_3\sqcup\cb_4\sqcup\cb_5$, due to the equalities \eqref{eqn:ca-I}, \eqref{eqn:ca-II} and \eqref{eqn:ca-III}.
\begin{itemize}
  \item[(1)] $\cb_1=\ca_t(\Lambda_n)$, because $\ca_s^{{\rm I},1}(\Gamma_n)=\ca_t(\Lambda_n)$ and both $\ca_s^{{\rm II},2}(\Gamma_n)$ and $\ca_s^{{\rm III},1}(\Gamma_n)$ are subsets of $\ca_t(\Lambda_n)$.
  \item[(2)] For $n\leq 4$, $\cb_2=\emptyset$, and for $n\geq 5$, $\cb_2$ consists of isoclasses of path algebras of elements in $\cm(n)$, by the analysis in Section~\ref{sss:family-II} Case 2.
  \item[(3)] For $\cb_3$, we first show that
  $\ca_t^6(\Lambda_n)=\ca_t^4(\Lambda_n)\cup \ca_s^{{\rm II},2}(\Gamma_n)$. 
For $n=2$, $\ca_t^4(\Lambda_2)=\ca_t(\Lambda_2)=\{\Lambda_2\}$ and $\ca_s^{\rm{II},2}(\Gamma_2)=\emptyset$, and hence $\ca_t^6(\Lambda_2)=\ca_t(\Lambda_2)=\ca_t^4(\Lambda_2)\cup\ca_s^{\rm{II},2}(\Gamma_2)$. 
For $n\geq 3$, look at quivers with relation obtained in the subcases (2.1) and (2.2) in Section~\ref{sss:family-II}. 
By Lemma~\ref{lem:tilting-modules-and-tilted-algebras-with-P(n)} (a), the only element of $\ca_s^{\rm{II},2}(\Gamma_n)$ which is the endomorphism algebra of a tilting module over $\Lambda_n$ with $P(n)$ as a direct summand is the path algebra of
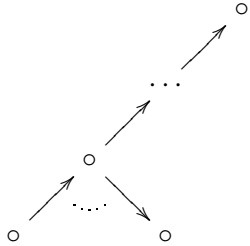
\[
\begin{xy}
(0,0)*+{\cdots}="1",
(10,10)*+{\circ}="2",
(-10,-10)*+{\circ}="4",
(-20,-20)*+{\circ}="5",
(0,-20)*+{\circ}="6",
(-14,-14)*+{ }="7",
(-6,-14)*+{ }="8",
\ar"5";"4",\ar"4";"1",\ar"1";"2", \ar"4";"6",\ar@/^0.6pc/@{.}"8";"7",
\end{xy}
\]
The tilting module is exactly $P(1)\oplus\ldots\oplus P(n-2)\oplus P(n)\oplus \tau^{-2}P(n)$. This shows that $\ca_t^6(\Lambda_n)=\ca_t^4(\Lambda_n)\cup \ca_s^{{\rm II},2}(\Gamma_n)$. Therefore
  \begin{align*}
\cb_3&=\ca_s^{\rm{I},2}(\Gamma_n)\cup\ca_s^{\rm{II},1}(\Gamma_n)\cup\bigcup_{i=2}^{n-1}(\ca_s^{\rm{II},2}(\Gamma_{n-i+1})\times_s \ca_t(\Lambda_{i-1}))\\
&=\bigcup_{m=2}^{n-1}(\ca_t^4(\Lambda_m)\times_s \ca_t(\Lambda_{n-m})) \cup(\ca_t^1(\Lambda_{n-1})\times_s\ca_t(\Lambda_1))\cup \bigcup_{m=2}^{n-1}(\ca_s^{\rm{II},2}(\Gamma_{m})\times_s \ca_t(\Lambda_{n-m}))\\
&= ((\ca_t^4(\Lambda_{n-1})\cup\ca_t^1(\Lambda_{n-1})\cup\ca_s^{\rm{II},2}(\Gamma_{n-1}))\times_s \ca_t(\Lambda_{1}))\cup\bigcup_{m=2}^{n-2}(\ca_t^6(\Lambda_{m})\times_s \ca_t(\Lambda_{n-m}))\\
&=((\ca_t(\Lambda_{n-1})\times_s \ca_t(\Lambda_{1}))\cup\bigcup_{m=2}^{n-2}(\ca_t^6(\Lambda_{m})\times_s \ca_t(\Lambda_{n-m}))\\
&=(\ca_t(\Lambda_{1})\times_s \ca_t(\Lambda_{n-1}))\cup (\ca_t^6(\Lambda_{n-1})\times_s \ca_t(\Lambda_{1}))\cup\bigcup_{m=2}^{n-2}(\ca_t^6(\Lambda_{m})\times_s \ca_t(\Lambda_{n-m}))\\
&=\bigcup_{m=1}^{n-1}(\ca_t^6(\Lambda_{m})\times_s \ca_t(\Lambda_{n-m})).
\end{align*}
  \item[-] For $\cb_4$,
  \begin{align*}
  \cb_4&=\bigcup_{i=2}^{n-1}(\ca_s^{\rm{II},3}(\Gamma_{n-i+1})\times_s \ca_t(\Lambda_{i-1}))=\bigcup_{i=2}^{n-1}(\cb_2(\Gamma_{n-i+1})\times_s \ca_t(\Lambda_{i-1}))\\
&=\bigsqcup_{i=2}^{n-1}(\cb_2(\Gamma_{n-i+1})\times \ca_t(\Lambda_{i-1}))=\bigsqcup_{m=5}^{n-1}(\cb_2(\Gamma_{m})\times \ca_t(\Lambda_{n-m})).
\end{align*}
The last equality holds because $\cb_2(\Gamma_m)=\emptyset$ for $m\leq 4$.
  \item[-] For $\cb_5$,
  \begin{align*}
  \cb_5&=\bigcup_{i=2}^{n-1}(\ca_t^{2}(\Lambda_{n-i})\times_s\ca_t(\Lambda_1)\times_s\ca_t(\Lambda_{i-1}))\\
  &=\bigcup_{m=1}^{n-2}(\ca_t^1(\Lambda_{m})\times_s \ca_t(\Lambda_{n-m-1}))\times \ca_t(\Lambda_1).
  \end{align*}
\end{itemize}
The proof of Theorem~\ref{thm:silted-algebras-of-Gamman} is complete.

\subsubsection{2-term tilting complexes over $\Gamma_n$}
As a consequence of the analysis in Subsections~\ref{sss:family-I}, \ref{sss:family-II} and \ref{sss:family-III}, we obtain a classification of basic 2-term tilting complexes over $\Gamma_n$.
Note that basic tilting modules over $\Gamma_n$ has been classified in the proof of Corollary~\ref{cor:number-of-tilted-algebra-of-tyep-Gamman}.

\begin{corollary}
A complex of $K^b(\proj \Gamma_n)$ is a basic 2-term tilting complex if and only if it is of one of the following forms:
\begin{itemize}
\item[-] $T$, a basic tilting module over $\Lambda_n$ which does not contain $P(n)$ as a direct summand,
\item[-] $\tau^{-1}T$, where $T$ is a basic tilting module over $\Lambda_n$,
\item[-] $T^*$, where $T$ is a basic tilting module over $\Lambda_n$ which contains $P(n)$ but not $P(n-1)$ as a direct summand,
\item[-] $\tau^{-2}(T)$, where $T$ is a basic tilting module over $\Lambda_n$ which contains $P(2)$ as a direct summand.
\end{itemize}
\end{corollary}

\subsection{The cardinalities of $\cb_i$, $1\leq i\leq 5$}
\label{ss:cardinatities}
In this subsection we prove Theorem~\ref{thm:number-of-silted-algebras-of-Gamman}. The equalities for $|\cb_1|$ and $|\cb_4|$ are clear. 
We only need to calculate the cardinalities of $\cb_2$, $\cb_3$ and $\cb_5$.

\subsubsection{The cardinality of $\cb_2$}
\label{sss:the-cardinality-of-B2}

Recall that the map $\mu\colon\cq^3(n)\to \cm(n)$ is bijective. It follows  from Lemma~\ref{lem:number-of-tilting-modules-without-P(n-1)-as-summand} that 
\begin{align}
\label{card:M}
|\cm(n)|=t(\Lambda_{n-1})-t(\Lambda_{n-2})-2^{n-2}+1.
\end{align}
Recall that $\cm(3)$ and $\cm(4)$ are empty and $\cm(5)$ has two elements, see~\eqref{QR:M(5)}. It follows that $|\cb_2|=0,0,2$ for $n=3,4,5$. In the sequel of this subsubsection we assume $n\geq 6$. The map $\cm(n)\to\cb_2$ of taking the path algebra is by definition surjective, but it is not injective, namely, different rooted quivers with relation in $\cm(n)$ can be isomorphic as quivers with relation. A typical example is
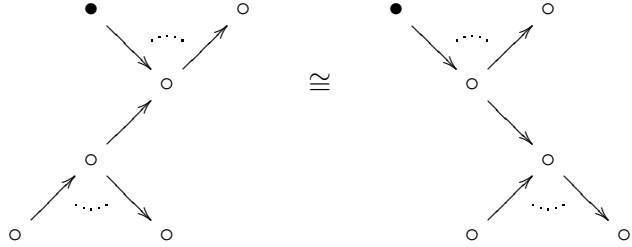
\[
\begin{xy}
 (-10,-10)*+{\circ}="11",
(0,0)*+{\circ}="22",
(-10,10) *+{\bullet}="21",
(10,10) *+{\circ}="23",
(-4,4)*+{ }="24",
(4,4)*+{ }="25",
(-20,-20) *+{\circ}="1",
(0,-20) *+{\circ}="3",
(-14,-14)*+{ }="4",
(-6,-14)*+{ }="5",
\ar"21";"22", \ar"22";"23", \ar@/_0.6pc/@{.}"25";"24",
\ar"11";"22", 
\ar"1";"11", \ar"11";"3", \ar@/^0.6pc/@{.}"5";"4",
\end{xy}
\hspace{15pt} \cong \hspace{15pt}
\begin{xy}
 (10,-10)*+{\circ}="11",
(0,0)*+{\circ}="22",
(-10,10) *+{\bullet}="21",
(10,10) *+{\circ}="23",
(-4,4)*+{ }="24",
(4,4)*+{ }="25",
(0,-20) *+{\circ}="1",
(20,-20) *+{\circ}="3",
(6,-14)*+{ }="4",
(14,-14)*+{ }="5",
\ar"21";"22", \ar"22";"23", \ar@/_0.6pc/@{.}"25";"24",
\ar"22";"11", 
\ar"1";"11", \ar"11";"3", \ar@/^0.6pc/@{.}"5";"4",
\end{xy}
\]
In order to compute the cardinality of $\cb_2$ we need to figure out which elements of $\cm(n)$ are isomorphic as quivers with relation. 

\smallskip
For $1\leq m\leq n-2$, consider the following rooted quiver with relation with $m+2$ vertices
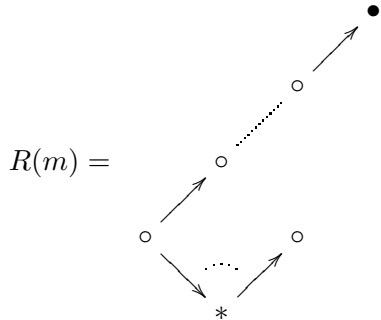
\[
R(m)=~~~\begin{xy}
(0,-10) *+{\circ}="21",
(20,-10) *+{\circ}="23",
(10,-20) *+{*}="22",
(6,-16)*+{ }="24",
(14,-16)*+{ }="25",
(10,0) *+{\circ}="26",
(20,10) *+{\circ}="27",
(30,20) *+{\bullet}="28",
\ar"21";"22", \ar"22";"23", \ar@/_0.6pc/@{.}"25";"24",
\ar@{.}"26";"27", \ar"21";"26", \ar"27";"28", 
\end{xy}
\]
For $1\leq m,l\leq n-2$ with $m+l\leq n-3$ and for $R\in\cq_h(n-2-m-l)$, define $G(m,l,R)$ to be the rooted quiver with relation obtained from $R(m),R(l)$ and $R$ by identifying the root of $R$ with the vertex $*$ of $R(m)$ and identifying the leaf of $R$ with the vertex $*$ of $R(l)$ (we turn $R(l)$ upside down). The root of $G(m,l,R)$ is the root of $R(m)$. For example, \begin{align*}
G(1,2,\begin{xy}
 (-7.5,-3)*+{\circ}="1",
(0,4.5)*+{\bullet}="2",
\ar"1";"2", 
\end{xy})=
\begin{xy}
(0,16) *+{\bullet}="1",
(20,16) *+{\circ}="3",
(10,6) *+{\circ}="2",
(6,10)*+{ }="24",
(14,10)*+{ }="25",
(0,-16) *+{\circ}="4",
(10,-6) *+{\circ}="5",
(20,-16) *+{\circ}="6",
(10,-26) *+{\circ}="7",
(6,-10)*+{ }="34",
(14,-10)*+{ }="35",
\ar"1";"2", \ar"2";"3", \ar"5";"2", \ar@/_0.6pc/@{.}"25";"24",
\ar"4";"5", \ar"5";"6", \ar@/^0.6pc/@{.}"35";"34",
\ar"4";"7",
\end{xy}
\end{align*}
Let 
\[
\mathcal{R}(m,l)=
\begin{cases} \{G(m,l,R)\mid R\in \cq_h(n-2-m-l)\} & \text{if } m\neq l,\\
\{G(m,m,R)\mid R\in\cq_h(n-2-2m),\bar{R}\neq R\} & \text{if }m=l.
\end{cases}
\]
Note that $\mathcal{R}(1,l)\subseteq \cm_1(n)$ and $\mathcal{R}(m,l)\subseteq \cm_2(n)$ for $m\geq 2$, and that $\mathcal{R}(m,m)$ is defined for $m\leq \frac{n-3}{2}$. Moreover, by Lemma~\ref{lem:hereditary-QR}, we have 
\begin{align}
\label{card:R}
|\mathcal{R}(m,m)|&=\begin{cases} 2^{n-3-2m}-2^{\frac{n-3}{2}-m} & \text{if $n$ is odd},\\
2^{n-3-2m} & \text{if $n$ is even},
\end{cases}\\
|\mathcal{R}(m,l)|&=2^{n-3-m-l} \hspace{10pt} \text{if $m\neq l$}.\nonumber
\end{align}

We claim that for $R\in \cm(n)$, the following conditions are equivalent:
\begin{itemize}
\item[(i)]
there exists $R'\in \cm(n)$ such that $R'\neq R$ but $R'$ is isomorphic to $R$ as a quiver with relation,
\item[(ii)] $R\in \mathcal{R}(m,l)$ for some $m$ and $l$, or $R$ is one of:
\begin{align}
\label{QR:exceptional-case}
\begin{xy}
(0,15) *+{\circ}="21",
(18,13) *+{\circ}="23",
(10,5) *+{\circ}="22",
(6,9)*+{ }="24",
(14,9)*+{ }="25",
(20,-5) *+{\circ}="26",
(30,-15) *+{\circ}="27",
(40,-25) *+{\circ}="28",
(18,17) *+{\circ}="33",
(10,25) *+{\bullet}="32",
(6,21)*+{ }="34",
(14,21)*+{ }="35",
\ar"21";"22", \ar"22";"23", \ar@/_0.6pc/@{.}"25";"24",
\ar"21";"32", \ar"32";"33", \ar@/^0.6pc/@{.}"35";"34",
\ar"22";"26", \ar@{.}"26";"27", \ar"27";"28",
\end{xy},
\hspace{30pt}
\begin{xy}
(0,-15) *+{\circ}="21",
(18,-13) *+{\circ}="23",
(10,-5) *+{\circ}="22",
(6,-9)*+{ }="24",
(14,-9)*+{ }="25",
(20,5) *+{\circ}="26",
(30,15) *+{\circ}="27",
(40,25) *+{\bullet}="28",
(18,-17) *+{\circ}="33",
(10,-25) *+{\circ}="32",
(6,-21)*+{ }="34",
(14,-21)*+{ }="35",
\ar"21";"22", \ar"22";"23", \ar@/^0.6pc/@{.}"25";"24",
\ar"21";"32", \ar"32";"33", \ar@/_0.6pc/@{.}"35";"34",
\ar"22";"26", \ar@{.}"26";"27", \ar"27";"28",
\end{xy}
\end{align}
\end{itemize} 
Moreover, the $R'$ in (i)  is unique: if $R=G(m,l,R_1)\in\mathcal{R}(m,l)$, then $R'=G(l,m,\bar{R}_1)$; if $R$ is one of the two rooted quivers with relation in \eqref{QR:exceptional-case}, then $R'$ is the other one. 
As a consequence, we have
\begin{align*}
|\cb_2|&=|\cm(n)|-\sum_{m=1}^{[\frac{n-3}{2}]} \frac{1}{2} |\mathcal{R}(m,m)|-\sum_{1\leq m<l\leq n-3-m}|\mathcal{R}(m,l)|-1\\
&=|\cm(n)|-\frac{1}{2}\sum_{i=2}^{n-3}\sum_{m=1}^{i-1}|\mathcal{R}(m,i-m)|-1.
\end{align*}
When $n$ is even, we have
\begin{align*}
\sum_{i=2}^{n-3}\sum_{m=1}^{i-1}|\mathcal{R}(m,i-m)|&=\sum_{i=2}^{n-3}\sum_{m=1}^{i-1}2^{n-3-i}
=\sum_{i=2}^{n-3}(i-1)2^{n-3-i}\\
&=\sum_{i=2}^{n-3}2^{n-3-i}+\sum_{i=3}^{n-3}2^{n-3-i}+\ldots+\sum_{i=n-2}^{n-3}2^{n-3-i}+1\\
&=(2^{n-4}-1)+(2^{n-5}-1)+\ldots+(2^2-1)+(2-1)\\
&=2^{n-3}-n+2.
\end{align*}
So
\begin{align*}
|\cb_2|&=t(\Lambda_{n-1})-t(\Lambda_{n-2})-2^{n-2}+1-\frac{1}{2}(2^{n-3}-n+2)-1\\
&=t(\Lambda_{n-1})-t(\Lambda_{n-2})-2^{n-2}-2^{n-4}+\frac{n-2}{2}.
\end{align*}
When $n$ is odd, we have
\begin{align*}
\sum_{i=2}^{n-3}\sum_{m=1}^{i-1}|\mathcal{R}(m,i-m)|&=\sum_{i=2}^{n-3}\sum_{m=1}^{i-1}2^{n-3-i}-\sum_{m=1}^{\frac{n-3}{2}}2^{\frac{n-3}{2}-m}\\
&=(2^{n-3}-n+2)-(2^{\frac{n-3}{2}}-1)\\
&=2^{n-3}-2^{\frac{n-3}{2}}-(n-3).
\end{align*}
So
\begin{align*}
|\cb_2|&=t(\Lambda_{n-1})-t(\Lambda_{n-2})-2^{n-2}+1-\frac{1}{2}(2^{n-3}-2^{\frac{n-3}{2}}-(n-3))-1\\
&=t(\Lambda_{n-1})-t(\Lambda_{n-2})-2^{n-2}-2^{n-4}+2^{\frac{n-3}{2}-1}+\frac{n-3}{2}.
\end{align*}

\smallskip
Next we prove the claim. The implication (ii)$\Rightarrow$(i) is clear with $R'$ as given in the `Moreover' part. Now let $R'\in\cm(n)$ be such that $R'\neq R$ and there are mutual inverse isomorphisms $\varphi\colon R\to R'$ and $\psi\colon R'\to R$ of quivers with relation.

Case 1: $R\in\cm_3(n)$. By Lemma~\ref{lem:non-hereditary-mutated-QR-3}, we assume that $R=E(R_1,R_2,R_3)\in\cm_3(n)$ for some $R_1\in\cq^1(j-i+1)$, $R_2\in\cq(n-j-1)$ and $R_3\in\cq(i-2)$, $3\leq i\leq j\leq n-2$. Then $R$ has a unique vertex (the vertex $2$) which lies on a relation and has exactly two neighbours. 

\emph{Subcase (1.1):} $R'\in\cm_1(n)$. By Lemma~\ref{lem:non-hereditary-mutated-QR-1}, we assume that $R'=E(R'_1)\in\cm_1(n)$ for some $R'_1\in\cq_{nh}(n-2)$. By Lemma~\ref{lem:vertices-lying-on-a-relation} applied to $R'_1$, we see that $R'$ can not have a vertex which has exactly two neighbours and which lies on a relation. So this subcase does not occur. 

\emph{Subcase (1.2):} $R'\in\cm_2(n)$. By Lemma~\ref{lem:non-hereditary-mutated-QR-2}, we assume that $R'=E(R'_1,R_2)\in\cm_2(n)$ for some $R'_1\in\cq^1(j'-1)$ and $R'_2\in\cq(n-j'-1)$ such that at least one of $R'_1$ and $R'_2$ is non-hereditary, $2\leq j'\leq n-2$. By Lemma~\ref{lem:vertices-lying-on-a-relation} applied to $R'_1$ and $R'_2$, we see that $R'$ has a vertex which has exactly two neighbours and which lies on a relation only if $R'_2$ consists of the root only, and thus $j'=n-3$ and $R'_1$ has to be non-hereditary. In this case, this vertex is the vertex $4$ in $R'$. The isomorphism $\varphi\colon R\to R'$ necessarily takes the vertex $2$ of $R$ to the vertex $4$ of $R'$, and hence takes the vertex $1$ of $R$ to the vertex $3$ of $R'$ and takes the vertex $3$ of $R$ to the vertex $5$ of $R'$, which is a sink. It follows that $R_3$ consists of the root only (so $i=3$).  Moreover, in $R$ the vertex $4$ is the only vertex which is a source and from which there are two arrows going out, and in $R'$ the vertex $3$ is the only such vertex. Therefore the isomorphism $\varphi\colon R\to R'$ has to take the vertex $4$ of $R$ to the vertex $3$ of $R'$. This implies that in $R$ the vertices $1$ and $4$ are the same and $R_1$ consists of the root only and that $R'$ is obtained from some $R''_1\in \cq^1(n-5)$ and
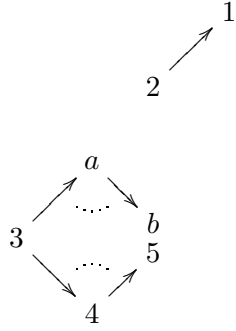
\[
\begin{xy}
(0,0) *+{3}="21",
(18,-2) *+{5}="23",
(10,-10) *+{4}="22",
(6,-6)*+{ }="24",
(14,-6)*+{ }="25",
(10,10) *+{a}="31",
(18,2) *+{b}="32",
(18,20) *+{2}="26",
(28,30) *+{1}="27",
(6,6)*+{ }="34",
(14,6)*+{ }="35",
\ar"21";"22", \ar"22";"23", \ar@/_0.6pc/@{.}"25";"24",
\ar"21";"31", \ar"31";"32",  \ar@/^0.6pc/@{.}"35";"34",
\ar"26";"27",
\end{xy}
\]
by identifying the root of $R''_1$ with the vertex $2$ and the source of the left border of $R''_1$ with $a$, and that by removing the vertex $3,4,5,b$ and taking $a$ as the root we obtain $R_2$. This implies that $R_2$ has to be the quiver $\overrightarrow{\mathbb{A}}_{n-4}$ with trivial relation and with the unique source as the root and $R''_1$ has to be the quiver $\overrightarrow{\mathbb{A}}_{n-5}$ with trivial relation and with the unique source as the root. Therefore $R$ is the first rooted quiver with relation in \eqref{QR:exceptional-case} and $R'$ is the second one.

\emph{Subcase (1.3):} $R'\in\cm_3(n)$. Then the isomorphism $\varphi\colon R\to R'$ has to take the vertex $2$ of $R$ to the vertex $2$ of $R'$, and hence take the vertex $1$ of $R$ to the vertex $1$ of $R'$ and the vertex $3$ of $R$ to the vertex $3$ of $R'$. Then it follows by induction that $R=R'$, a contradiction.

Case 2: $R\in\cm_2(n)$. By Lemma~\ref{lem:non-hereditary-mutated-QR-2}, we assume that $R=E(R_1,R_2)\in\cm_2(n)$ for some $R_1\in\cq^1(j-1)$, $R_2\in\cq(n-j-1)$, and at least one of $R_1$ and $R_2$ is non-hereditary, $2\leq j\leq n-2$. 

\emph{Subcase (2.1):} $R'\in\cm_1(n)$. By Lemma~\ref{lem:non-hereditary-mutated-QR-1}, we assume that $R'=E(R'_1)\in\cm_1(n)$ for some $R'_1\in \cq_{nh}(n-2)$. Let $a,b,c$ be the images of $1,2,3$ under the isomorphism $\psi\colon R'\to R$. Suppose that $a,b,c$ belong to $R_1$. By removing the vertices $a$ and $c$ and taking $b$ as the root we obtain $R'_1\in\cq(n-2)$. By the same argument as in the proof of Lemma~\ref{lem:non-hereditary-mutated-QR-2} (a), this implies that $R_2$ is hereditary, and that the rooted quiver with relation $R_3$ obtained from $R_1$ by removing the vertices $a$ and $c$ is also hereditary. It follows that $R_3$ has to be of the form $\overrightarrow{\mathbb{A}}_{m}$ for some $m$, but none of the vertices of this quiver can serve as $b$. Therefore $a,b,c$ cannot belong to $R_1$, and hence they have to belong to $R_2$. We deduce as above that $R_1$ has to be hereditary and the rooted quiver with relation $R_4$ obtained from $R_2$ by removing $a$ and $c$ also has to be hereditary. It follows that $R=G(j,1,R_4)$, and $R'=G(1,j,\bar{R}_4)$.

\emph{Subcase (2.2):} $R'\in\cm_2(n)$. By Lemma~\ref{lem:non-hereditary-mutated-QR-2}, we assume that $R'=E(R'_1,R'_2)\in\cm_2(n)$ for some $R'_1\in\cq^1(j'-1)$, $R_2\in\cq(n-j'-1)$, and at least one of $R'_1$ and $R'_2$ is non-hereditary, $2\leq j'\leq n-2$. If $R_1$ is non-hereditary, then $R$ has a full subquiver with relation of the form \eqref{QR:type-box}. Note that $R_2$ does not contribute to such a quiver with relation, so the lower part of \eqref{QR:type-box} has to be the full subquiver of $R$ consisting of the vertices $3$, $4$ and $5$, and this implies that the isomorphism $\varphi\colon R\to R'$ has to take the vertices $3$, $4$ and $5$ of $R$ to the vertices $3$, $4$ and $5$ of $R'$. It follows by induction that $R=R'$. Therefore $R_1$ has to be hereditary and hence as a quiver it has to be $\overrightarrow{\mathbb{A}}_{j-1}$. This means that the top part of $R$ is of the form $R(j)$, and its image under the isomorphism $\varphi\colon R\to R'$ lands in $R'_2$. Similarly, the top part of $R'$ is of the form $R(j')$ and its image under the isomorphism $\psi\colon R'\to R$ lands in $R_2$. Let $R_4$ (respectively, $R'_4$) be the rooted quiver with relation obtained from $R_2$ (respectively, $R'_2$) by removing the image of $R(j')$ (respecitively, $R(j)$). Then the isomorphism $\varphi\colon R\to R'$ induces an isomorphism $R_4\to R'_4$ of quivers with relation. This implies that we can find another vertex of $R_4$ as the root to make $R_4$ a rooted quiver with relation. Therefore $R_4,R'_4\in\cq_h(n-j-j'-4)$, $R'_2=\bar{R}_2$, $R=G(j,j',R_4)$ and $R'=G(j',j,R'_4)$. 

Case 3: $R,R'\in\cm_1(n)$. By Lemma~\ref{lem:non-hereditary-mutated-QR-1}, we assume that $R=E(R_1)$ and $R'=E(R'_1)$ for some $R_1,~R'_1\in\cq_{nh}(n-2)$. If the isomorphism $R\to R'$ of quivers with relation perserves the vertices $1$, $2$ and $3$, then it follows by induction that $R=R'$ and in particular $R_1=R'_1$. Thus we assume that $R_1\neq R'_1$. Let $a,b,c$ be the images of $1,2,3$ under the isomorphism $\varphi$ (respectively, $\psi$), and let $R_2$ (respectively, $R'_2$) be the rooted quiver with relation obtained from $R_1$ (respectively, $R'_1$) by removing the vertices $a$ and $c$. Then the isomorphism $\varphi\colon R\to R'$ induces an isomorphism $R_2\to R'_2$ of quivers with relation. This implies that we can find another vertex of $R_2$ as the root to make $R_2$ a rooted quiver with relation. Therefore $R_2,R'_2\in\cq_h(n-4)$, $R'_2=\bar{R}_2$, $R=G(1,1,R_2)$ and $R'=G(1,1,R'_2)$.

\subsubsection{The cardinality of $\cb_3$}
We have
\begin{align*}
\cb_3&=\bigcup_{m=1}^{n-1}(\ca_t^6(\Lambda_{m})\times_s \ca_t(\Lambda_{n-m}))\\
&=\bigsqcup_{m=1}^{[\frac{n}{2}]}(\ca_t^6(\Lambda_m)\times_s \ca_t(\Lambda_{n-m}))\cup (\ca_t^6(\Lambda_{n-m})\times_s \ca_t(\Lambda_{m}))
\end{align*}
For $1\leq m\leq [\frac{n}{2}]$, put
\[
\cb_{3,m}=
(\ca_t^6(\Lambda_m)\times_s \ca_t(\Lambda_{n-m}))\cup (\ca_t^6(\Lambda_{n-m})\times_s \ca_t(\Lambda_{m})) ~\text{and}~
b_{3,m}=|\cb_{3,m}|.
\]
Then
\[
\cb_3=\bigsqcup_{m=1}^{[\frac{n}{2}]}\cb_{3,m}~\text{and}~
|\cb_3|=\sum_{m=1}^{[\frac{n}{2}]} b_{3,m}.
\]

If $m=1,2,3$ and $m<\frac{n}{2}$, then $\cb_{3,m}=\ca_t(\Lambda_{m})\times_s\ca_t(\Lambda_{n-m})=\ca_t(\Lambda_{m})\times\ca_t(\Lambda_{n-m})$ because $\ca_t^6(\Lambda_m)=\ca_t(\Lambda_m)$, so
\begin{align*}
b_{3,m}&=|\cb_{3,m}|=a_t(\Lambda_m)a_t(\Lambda_{n-m}).
\end{align*}
If $4\leq m<\frac{n}{2}$, then it follows from Lemma~\ref{lem:symmetric-product} that
\begin{align*}
b_{3,m}&=|\cb_{3,m}|=|(\ca_t^6(\Lambda_m)\times_s \ca_t(\Lambda_{n-m}))\cup (\ca_t^6(\Lambda_{n-m})\times_s \ca_t(\Lambda_{m}))|\\
&=a_t^6(\Lambda_m)a_t(\Lambda_{n-m})+a_t^6(\Lambda_{n-m})a_t(\Lambda_m)-a_t^6(\Lambda_m)a_t^6(\Lambda_{n-m}).
\end{align*}
If $n$ is even and $m=\frac{n}{2}$, then $\cb_{3,\frac{n}{2}}=\ca_t^6(\Lambda_{\frac{n}{2}})\times_s \ca_t(\Lambda_{\frac{n}{2}})$ and it follows from Lemma~\ref{lem:symmetric-product} that
\begin{align*}
b_{3,\frac{n}{2}}&=|\cb_{3,\frac{n}{2}}|
=a_t^6(\Lambda_{\frac{n}{2}})a_t(\Lambda_{\frac{n}{2}})-\frac{a_t^6(\Lambda_{\frac{n}{2}})(a_t^6(\Lambda_{\frac{n}{2}})-1)}{2}.
\end{align*}

\subsubsection{The cardinality of $\cb_5$} We have
\begin{align*}
  \cb_5  &=\bigcup_{m=1}^{n-2}(\ca_t^1(\Lambda_{m})\times_s \ca_t(\Lambda_{n-m-1}))\times \ca_t(\Lambda_1)\\
  &=\bigsqcup_{m=1}^{[\frac{n-1}{2}]}(\ca_t^1(\Lambda_{m})\times_s \ca_t(\Lambda_{n-m-1}))\cup(\ca_t^1(\Lambda_{n-m-1})\times_s \ca_t(\Lambda_{m})))\times\ca_t(\Lambda_1).
\end{align*}
Put
\[
\cb_{5,m}=(\ca_t^1(\Lambda_{m})\times_s \ca_t(\Lambda_{n-m-1}))\cup(\ca_t^1(\Lambda_{n-m-1})\times_s\ca_t(\Lambda_m)) \text{ and } b_{5,m}=|\cb_{5,m}|.
\]
Then
\begin{align*}
\cb_5&=\bigsqcup_{m=1}^{[\frac{n-1}{2}]}\cb_{5,m}\times\ca_t(\Lambda_1),\\
|\cb_5|&=\sum_{m=1}^{[\frac{n-1}{2}]} b_{5,m}.
\end{align*}

If $m=1$, then $\cb_{5,1}=\ca_t(\Lambda_{n-2})$ and $b_{5,1}=a_t(\Lambda_{n-2})$. This can be merged into the next case if we set $t(\Lambda_0)=1$.  
If $1<m< \frac{n-1}{2}$, then it follows from Lemma~\ref{lem:symmetric-product} and Lemma~\ref{lem:tilting-modules-and-tilted-algebras-with-P(n)} that
\begin{align*}
b_{5,m}&=a_t^1(\Lambda_m)a_t(\Lambda_{n-m-1})+a_t^1(\Lambda_{n-m-1})a_t(\Lambda_m)-a_t^1(\Lambda_m) a_t^1(\Lambda_{n-m-1})\\
&=t(\Lambda_{m-1})a_t(\Lambda_{n-m-1})+t(\Lambda_{n-m-2})a_t(\Lambda_m)-t(\Lambda_{m-1})t(\Lambda_{n-m-2}).
\end{align*}
If $n$ is odd and $m=\frac{n-1}{2}$, then it follows from Lemma~\ref{lem:symmetric-product} and Lemma~\ref{lem:tilting-modules-and-tilted-algebras-with-P(n)} that 
\begin{align*}
b_{5,\frac{n-1}{2}}
&=
a_t^1(\Lambda_{\frac{n-1}{2}})a_{t}(\Lambda_{\frac{n-1}{2}})-\frac{a_t^1(\Lambda_{\frac{n-1}{2}})\times(a_t^1(\Lambda_{\frac{n-1}{2}})-1)}{2}\\
&=
t(\Lambda_{\frac{n-3}{2}})a_{t}(\Lambda_{\frac{n-1}{2}})-\frac{t(\Lambda_{\frac{n-3}{2}})\times(t(\Lambda_{\frac{n-3}{2}})-1)}{2}.
\end{align*}

\section{No strictly shod algebras in type $\mathbb{A}$}
\label{s:no-strictly-shod-algebra-of-type-A}

As we have observed, as a consequence of Theorems~\ref{thm:the-number-of-silted-algebras-of-Lambda_n} and~\ref{thm:silted-algebras-of-Gamman}, that among the silted algebras of type $\Lambda_n$ and of type $\Gamma_n$ there are no strictly shod algebras. During the preparation of this paper, the third author and his collaborator showed in \cite{ZhangLiu22}, using geometric models, that there are no strictly shod algebras among silted algebra of type the path algebra of any quiver of type $\mathbb{A}$. Here we give an alternative proof. Assume that $k$ is algebraically closed.

\begin{theorem}[{\cite[Corollary 1.5]{ZhangLiu22}}]
\label{thm:no-strictly-shod-algebra-of-type-A}
Let $A$ be the path algebra of any fixed quiver of type $\mathbb{A}_n$. Then any silted algebra of type $A$ is a tilted algebra of type $\mathbb{A}_{m_1}\times\cdots\times\mathbb{A}_{m_r}$, where $m_1,\ldots,m_r$ are positive integers such that $m_1+\ldots+m_r=n$.
\end{theorem}

We first establish the following result.

\begin{lemma}
\label{lem:reduction}
Let $Q$ be any finite quiver without oriented cycles, and $N$ be a 2-term presilting complex in $K^b(\proj kQ)$. Then there exists a finite quiver $Q'$ together with a triangle equivalence $\thick(N)\to K^b(\proj kQ')$ which takes $N$ to a 2-term silting complex of $K^b(\proj kQ')$. If $Q$ is of type $\mathbb{A}$ and $\thick(N)$ is connected, so is $Q'$.
\end{lemma}
\begin{proof}
We may assume that $N$ is basic. By \cite[Lemma 3.10]{AiharaIyama12}, we can write $N=N_1\oplus\ldots\oplus N_r$ such that $(N_1,\ldots,N_r)$ is an exceptional sequence in $K^b(\proj kQ)$. Let $\ca=\{H^0X\mid X\in\thick(N)\}$.
Then by \cite[Proposition 6.6]{Krause12} and \cite[Theorem 5.1]{Bruening07} $\ca$ is a wide subcategory of $\mod kQ$ and $\thick(N)=\cd^b_{\ca}(\mod kQ)$, the full subcategory of $\cd^b(\mod kQ)$ consising of complexes whose cohomologies belong to $\ca$.  Moreover, $\ca$ has finitely many isoclasses of simple objects and the Loewy lengths of objects in $\ca$ have a uniform upper bound, and hence it has a projective generator and is finitely generated. By \cite[Corollary 2.21 and Corollary 2.22]{IngallsThomas09}, $\ca$ has a projective generator $P_\ca$ such that $\End(P_\ca)\cong kQ'$ for some finite quiver $Q'$. Since both $\mod kQ$ and $\ca$ are hereditary, it follows that the canonical triangle functor $\cd^b(\ca)\to\cd^b_\ca(\mod kQ)$ is a triangle equivalence. Therefore we obtain a chain of triangle equivalences
\[
\thick(N)=\cd^b_\ca(\mod kQ)\rightarrow\cd^b(\ca)\rightarrow K^b(\proj kQ').
\]
Write $N=N'\oplus P_N[1]$ with $N'\in\mod kQ$ and $P_N\in\proj kQ$. Then both $N'$ and $P_N$ belong to $\ca$ and $P_N$ is projective also in $\ca$. It follows that the image of $N$ under the above composed equivalence is a 2-term silting complex in $K^b(\proj kQ')$.

Now assume that $Q$ is of type $\mathbb{A}$ and $\thick(N)$ is connected. Then $Q'$ is connected, and the Hom-space between two indecomposable $kQ'$-modules is either $k$ or $0$. Then $Q'$ is of type $\mathbb{A}$.
\end{proof}

\begin{proof}[Proof of Theorem~\ref{thm:no-strictly-shod-algebra-of-type-A}]
Let $M$ be a 2-term silting complex over $A$. There are two cases.

Case 1: $\End(M)$ is connected. As Adam-Christiaan van Roosmalen observed (see for example \cite[Section 5.1]{Yang20}), this implies that $M$ is a tilting complex. Thus by Corollary~\ref{cor:endomorphism-algebra-of-2-term-tilting-is-tilted}, $\End(M)$ is a tilted algebra of type $\mathbb{A}_n$.

Case 2: $\End(M)$ is not connected. Let $N$ be any direct summand of $M$ such that $\End(N)$ is an indecomposable direct summand of $\End(M)$. Then by Lemma~\ref{lem:reduction} there is a finite quiver $Q'$ of type $\mathbb{A}_m$ for $m=|N|$ and a 2-term silting complex in $K^b(\proj kQ')$ whose endomorphism algebra is isomorphic to $\End(N)$. Therefore by the conclusion in Case 1 applied to this 2-term silting complex in $K^b(\proj kQ')$, we obtain that $\End(N)$ is a tilted algebra of type $\mathbb{A}_m$. This completes the proof.
\end{proof}


\def\cprime{$'$}
\providecommand{\bysame}{\leavevmode\hbox to3em{\hrulefill}\thinspace}
\providecommand{\MR}{\relax\ifhmode\unskip\space\fi MR }
\providecommand{\MRhref}[2]{%
  \href{http://www.ams.org/mathscinet-getitem?mr=#1}{#2}
}
\providecommand{\href}[2]{#2}

\end{document}